\newif\ifelsevier
\IfFileExists{elsarticle.cls}{\elseviertrue}{\elsevierfalse}

\ifelsevier
  \documentclass[3p]{elsarticle}
  \biboptions{numbers,sort&compress}
\else
  \documentclass[11pt,a4paper]{article}
  \usepackage[margin=2.5cm]{geometry}
  \usepackage[numbers,sort&compress]{natbib}
\fi

\usepackage[T1]{fontenc}
\usepackage[utf8]{inputenc}
\usepackage{amsmath,amssymb,amsthm}
\usepackage{booktabs}
\usepackage{multirow}
\usepackage{graphicx}
\usepackage[hidelinks]{hyperref}

\usepackage{float}
\floatstyle{ruled}
\newfloat{algo}{htbp}{loa}
\floatname{algo}{Algorithm}
\DeclareMathOperator{\trace}{tr}
\newcommand{\KwIn}[1]{\noindent\textbf{Input:} #1\par}
\newcommand{\KwOut}[1]{\noindent\textbf{Output:} #1\par\vspace{3pt}}

\theoremstyle{plain}
\newtheorem{theorem}{Theorem}
\newtheorem{proposition}[theorem]{Proposition}
\newtheorem{corollary}[theorem]{Corollary}
\newtheorem{lemma}[theorem]{Lemma}
\theoremstyle{definition}
\newtheorem{definition}[theorem]{Definition}
\newtheorem{example}{Example}
\theoremstyle{remark}
\newtheorem{remark}[theorem]{Remark}

\newcommand{\pinv}[1]{#1^{\dagger}}
\newcommand{\R}{\mathbb{R}}
\newcommand{\C}{\mathbb{C}}
\newcommand{\Q}{\mathbb{Q}}
\newcommand{\Rat}{\R(u_1,\dots,u_\ell)}
\newcommand{\Kat}{K(u_1,\dots,u_\ell)}
\newcommand{\rank}{\operatorname{rank}}
\newcommand{\Ima}{\operatorname{Im}}
\newcommand{\Ker}{\operatorname{Ker}}
\newcommand{\num}{\operatorname{num}}
\newcommand{\den}{\operatorname{den}}

\newcommand{\titletext}{Effective computation of Moore--Penrose inverses over
fields of rational functions by specializations}

\newcommand{\abstracttext}{%
In this paper we consider matrices whose entries are rational functions of several
parameters over a Moore--Penrose field, that is, a field with an involutory
automorphism over which every matrix has Moore--Penrose inverse. 
We prove that over any such field the Penrose conditions can be solved with
linear algebra alone, even though they form a polynomial system of degree two. We
then
bound the degrees of the numerator and of the denominator of the pseudoinverse in
terms of the degree of the entries and of the rank, and not of the dimensions of the
matrix. 
Specialization is the thread: the closed form is computed once over the field of
rational functions, and the question is at which parameter values it still
returns the pseudoinverse of the specialized matrix. We determine the values
where it does not, from the matrix alone and before computing the pseudoinverse.
They are the zeros of a polynomial built from the maximal minors, and they are
the values at which the rank decreases. This turns previous sufficient conditions into an exact
characterization.
The theory also gives symbolic algorithms for the real, the complex and the
parametric case, which we implement in Maple and test on $972$ timed
replications. We also apply them to a Leontief economic model, where the excluded
value is the point at which the economy stops being viable.}

\begin{document}

\ifelsevier
\begin{frontmatter}

\title{\titletext}

\author[cunef]{J. Sendra\corref{cor1}}
\ead{juana.sendra@cunef.edu}
\cortext[cor1]{Corresponding author.}

\address[cunef]{Dept. de Matem\'aticas, CUNEF Universidad, C. de Almansa, 101, 28040, Madrid, Spain}

\begin{abstract}
\abstracttext
\end{abstract}

\begin{keyword}
Moore--Penrose inverse \sep symbolic computation \sep rational function field
\sep anisotropic hermitian form \sep degree bound \sep specialization \sep
Leontief input--output model
\MSC[2020] 15A09 \sep 65F20 \sep 15A15 \sep 68W30 \sep 65Y20 \sep 91B66
\end{keyword}

\end{frontmatter}
\else
\title{\titletext}
\author{J. Sendra\\ CUNEF Universidad, Madrid, Spain\\
\texttt{juana.sendra@cunef.edu}}
\date{\today}
\maketitle
\begin{abstract}
\abstracttext
\end{abstract}
\fi

\section{Introduction}
\label{sec:intro}

In many structural models the coefficients of a linear system are estimated and not
fixed. If we treat the uncertain entries as parameters $u_{1},\dots,u_{\ell}$ and
compute over $\Q(u_{1},\dots,u_{\ell})$, we obtain a closed form that can be evaluated
at particular values of the parameters, differentiated exactly, and examined for
degeneracy. When the system is rectangular or rank deficient, the object we need is the
Moore--Penrose inverse and not the ordinary inverse
\cite{BenIsraelGreville2003,Penrose1955}.

This setting raises three difficulties, and they are of different kinds.
The first is algebraic. The four conditions that define the pseudoinverse are a
polynomial system in the unknown entries. Three of them are linear, but the fourth is
quadratic, so the problem is not one of linear algebra and elimination methods have
been used to solve it.

The second is a matter of cost. Over a field of rational functions each arithmetic
operation requires a common denominator, a greatest common divisor and a
normalization, so the price of one operation is not constant: it grows with the degrees
of the operands. Any cost model therefore needs a bound on the size of the result.

The third concerns the use of the closed form. Once we have the pseudoinverse as a
matrix of rational functions, we want to evaluate it at particular values of the
parameters. This is not always legitimate: at some values the closed form no longer
represents the pseudoinverse of the evaluated matrix. Those values must be known, and
preferably before the pseudoinverse has been computed.

\subsection{Background}

The pseudoinverse was introduced by Penrose \cite{Penrose1955} and its theory is
collected in \cite{BenIsraelGreville2003}. Over fields other than the real and the
complex numbers, existence is not automatic. Gabriel \cite{Gabriel1969} studied the
inverse over an arbitrary field, Bhaskara Rao \cite{BhaskaraRao1983} and Bapat,
Bhaskara Rao and Prasad \cite{BapatBhaskaraPrasad1990} characterized existence over
integral domains by a condition on the maximal minors, and Sendra and Sendra
\cite{SendraSendra2017} introduced the Moore--Penrose fields, those over which every
matrix has a pseudoinverse, and proved that the property is inherited by the field of
rational functions. A determinantal representation of the pseudoinverse in terms of
the maximal minors goes back to Bruening \cite{Bruening1987} for the full rank case
and to \cite{BapatBhaskaraPrasad1990} over integral domains; over the real numbers the
same object appears as the squared volume of the matrix in the sense of Ben-Israel
\cite{BenIsrael1992}.

For matrices with polynomial or rational function entries, the Leverrier--Faddeev
algorithm was extended by Karampetakis \cite{KarampetakisLAA1997}, and evaluation and
interpolation techniques were developed by Vologiannidis and Karampetakis
\cite{VologiannidisKarampetakis2004} and by Stanimirovi\'c and Petkovi\'c
\cite{StanimirovicPetkovic2006}; the algebraic structure of the result is studied in
\cite{KafetzisKarampetakis2022}. Degree bounds appear in that same line of work,
stated per variable. In the symbolic setting over general fields, Sendra and Sendra
compute Drazin inverses with rational function entries using Gr\"obner bases and
specialization \cite{SendraSendra2015,SendraSendra2016}, and
\cite[\S4]{SendraSendra2017} treats the Penrose equations as a polynomial system of
degree two; this last work is our starting point.

The geometry of the solution sets of the Penrose conditions over a general field was
studied by Stanimirovi\'c et al.\ \cite{StanimirovicLMA2021}, who determine the
dimension of the set of inner inverses and describe the sets of outer inverses as
affine algebraic varieties; the dimension of the set of inner inverses that we use in
Section~\ref{sec:theory} is theirs. In \cite{Stanimirovic2021} the same authors
represent outer inverses with prescribed range and null space in terms of inner
inverses, obtain them as solutions of linear matrix equations, and give sufficient
conditions for those representations to specialize properly; the exact
characterization of Section~\ref{sec:degrees} sharpens conditions of that kind.
The affine geometry of the sets of outer inverses is explored computationally in
\cite{SendraMT2026}. Symbolic computation of generalized inverses with computer
algebra systems, and an application to dynamical systems, is treated in
\cite{StanimirovicCAI2019}.

\subsection{Contributions}

We address the three difficulties in that order. First, we show that the whole
system can be solved with linear algebra. The three linear conditions define an
affine subspace, and on that subspace the residual of the fourth one is the
displacement from its unique solution, so the fourth condition is affine-linear
there and two linear solves suffice. The identity behind this is classical over
$\C$; what we prove is its residual form, and over an arbitrary Moore--Penrose
field. It is the residual form, and not the identity, that gives the size of the
second system and makes the elimination linear, and in the symbolic computation
literature these systems have been solved with Gr\"obner bases instead.
Second, we bound the degrees of the numerator and of the denominator
of the pseudoinverse in terms of the degree of the entries and of the rank alone, and
not of the dimensions of the matrix. Third, and this is our main result, we determine the parameter values that have to
be excluded when the closed form is specialized, that is, the values at which it
stops representing the pseudoinverse of the specialized matrix. They are the zeros
of a single polynomial built from the maximal minors, and they are also the values
at which the rank decreases. So we obtain that set from the matrix itself, before
computing the pseudoinverse. We implement the method in Maple and evaluate it on
$972$ timed replications.

Section~\ref{sec:theory} develops the algebraic reduction,
Section~\ref{sec:degrees} the degree and specialization results,
Sections~\ref{sec:algorithms}--\ref{sec:cost} the algorithms and the experiments, and
Section~\ref{sec:application} an application to a parametric Leontief model.

\section{The Penrose system over a parametric field}
\label{sec:theory}

\subsection{Notation}
\label{ssec:notation}

Throughout, $(K,\varphi)$ denotes a field together with an involutory automorphism
$\varphi$, that is, one with $\varphi\circ\varphi=\mathrm{id}$, and for a matrix $A$
over $K$ we write $A^{*}$ for the transpose of $\varphi(A)$. The two familiar cases
are $K=\C$ with complex conjugation, where $A^{*}$ is the conjugate transpose, and a
field with the identity involution, such as $\R$, where $A^{*}=A^{\mathsf T}$.
Section~\ref{ssec:mpfields} singles out the pairs $(K,\varphi)$ over which the
pseudoinverse always exists.

With $(K,\varphi)$ fixed, $A\in K^{m\times n}$ has $\rank(A)=r$, and $X=(x_{ij})$
denotes a matrix of $nm$ indeterminates over $K$. The symbol $\ell$ always denotes
the number of parameters, so that the corresponding parametric field is
$K(u_{1},\dots,u_{\ell})$, whereas $r$ is reserved for the rank; arbitrary lengths of
vectors and index sets are denoted by $s$. A Moore--Penrose inverse of $A$ is a
matrix $X$ satisfying
\begin{align}
  AXA &= A, \tag{P1}\label{eq:P1}\\
  XAX &= X, \tag{P2}\label{eq:P2}\\
  (AX)^{*} &= AX, \tag{P3}\label{eq:P3}\\
  (XA)^{*} &= XA. \tag{P4}\label{eq:P4}
\end{align}
Whenever it exists it is unique \cite{Penrose1955}, and we denote it by $\pinv{A}$.
We use the standard notation $A\{i,j,\dots\}$ for the set of matrices satisfying the
corresponding subset of the Penrose conditions, so that $\pinv{A}$ is the unique
element of $A\{1,2,3,4\}$. We regard \eqref{eq:P1}--\eqref{eq:P4} as a system of
polynomial equations in $K[x_{11},\dots,x_{nm}]$. For index sets $I$ and $J$ we
write $A[I,J]$ for the submatrix of $A$ on the rows $I$ and the columns $J$, and
we abbreviate it as $A[I,:]$ or $A[:,J]$ when all the columns, respectively all
the rows, are taken.

For the parametric setting of Section~\ref{sec:degrees} onwards, write
$P=K[u_{1},\dots,u_{\ell}]$, so that $\Kat$ is its field of fractions. The
\emph{numerator} and \emph{denominator} of a nonzero $f\in\Kat$ are the coprime
polynomials $p,q\in P$ with $f=p/q$, uniquely determined up to multiplication by a
common nonzero unit, and for a matrix $A=(a_{ij})\in\Kat^{m\times n}$ we define
\[
  \den(A)=\operatorname{lcm}\bigl\{\den(a_{ij})\bigr\},
  \qquad
  \num(A)=\den(A)\,A .
\]
So $\den(A)\in P$ is defined up to multiplication by a nonzero unit, and
$\num(A)$ is a matrix over $P$. Throughout, \emph{degree} refers only to polynomials and never to rational
functions. For $p\in P$ it is the
total degree, and for a matrix over $P$ it is the maximum total degree of its
entries. From
Section~\ref{sec:degrees} onwards $\delta$ is reserved for an upper bound on the
degrees of the entries of a polynomial matrix $A$.

\subsection{Moore--Penrose fields}
\label{ssec:mpfields}

The results in this subsection are taken from Sendra and Sendra
\cite{SendraSendra2017} and are recalled here in the form needed later.

The pair $(K,\varphi)$ of Section~\ref{ssec:notation} is called a
\emph{Moore--Penrose field} if $\pinv{A}$ exists for every $A\in K^{m\times n}$.

\begin{theorem}[{\cite[Thms.~1, 4 and 12]{SendraSendra2017}}]
\label{thm:mpfield}
The Moore--Penrose inverse $\pinv{A}$ exists if and only if
\[
  \rank(AA^{*})=\rank(A)=\rank(A^{*}A).
\]
Moreover, $(K,\varphi)$ is a Moore--Penrose field if and only if, for every
$s\ge1$,
\[
  \sum_{k=1}^{s}a_{k}\varphi(a_{k})=0,\qquad a_{k}\in K,
  \qquad\text{implies}\qquad a_{1}=\dots=a_{s}=0 .
\]
If $(K,\varphi)$ is a Moore--Penrose field, then so is
$(K(u_{1},\dots,u_{\ell}),\varphi_{e})$, where $\varphi_{e}$ denotes the natural
extension of $\varphi$ to the rational-function field.
\end{theorem}

\begin{corollary}
\label{cor:Ruexists}
$(\R,\mathrm{id})$ is a Moore--Penrose field, and therefore so is
$(\Rat,\mathrm{id})$. Hence every $A\in\Rat^{m\times n}$ has a Moore--Penrose
inverse, and $\pinv{A}\in\Rat^{n\times m}$: the pseudoinverse of a parametric matrix
is again a matrix of rational functions in the same parameters.
\end{corollary}

This follows immediately from Theorem~\ref{thm:mpfield} and is recorded explicitly
because it is the setting used throughout the remainder of the paper. For the
identity involution the criterion becomes $\sum_{k}a_{k}^{2}=0\Rightarrow a_{k}=0$
for all $k$, which holds over $\R$, while the extension to $\Rat$ follows from the
last statement of the theorem.

The hypothesis is restrictive, so the parametric setting is not merely a formal
generalization. Characteristic zero is necessary \cite[Cor.~5]{SendraSendra2017} but
not sufficient: $\Q(\sqrt2)$ with $a+b\sqrt2\mapsto a-b\sqrt2$ has characteristic
zero and is not a Moore--Penrose field \cite[Ex.~6]{SendraSendra2017}.

We write $K_{\varphi}=\{a\in K:\varphi(a)=a\}$ for the subfield of the elements
that $\varphi$ fixes, which is $\R$ for $(\C,\overline{\phantom{a}})$ and all of
$K$ for the identity involution. The parameters can only take values in
$K_{\varphi}$, because only then does evaluation commute with the involution, and
that is what the next theorem requires.

\begin{theorem}[specialization; {\cite[Thms.~13 and 15]{SendraSendra2017}}]
\label{thm:specialization}
Let $A\in K(u_{1},\dots,u_{\ell})^{m\times n}$ be such that $\pinv{A}$ exists, and
let $\alpha\in K_{\varphi}^{\,\ell}$ satisfy
\begin{equation}
  \label{eq:speccond}
  \operatorname{den}(A)(\alpha)\cdot\operatorname{den}(\pinv{A})(\alpha)\neq 0 .
\end{equation}
Then $\pinv{\bigl(A|_{u=\alpha}\bigr)}$ exists and equals
$\bigl(\pinv{A}\bigr)\big|_{u=\alpha}$.
\end{theorem}

\begin{definition}
\label{def:exceptional}
The \emph{exceptional set} $\mathcal{E}(A)$ is the set of parameter values $\alpha$
for which \eqref{eq:speccond} fails, together with the values lying outside
$K_{\varphi}^{\,\ell}$.
\end{definition}

We emphasize two features of this result. First, apart from
$\alpha\in K_{\varphi}^{\,\ell}$, the only hypothesis is that the two denominators
in \eqref{eq:speccond} do not vanish. We do not assume that the rank is preserved;
that follows as a consequence.

Second, the point $\alpha$ has to lie in $K_{\varphi}^{\,\ell}$, which for
$(\C,\overline{\phantom{a}})$ means real parameter values. This is not a technical
restriction. Conditions \eqref{eq:P3} and \eqref{eq:P4} are defined through the
involution, and $\varphi_{e}$ conjugates coefficients and fixes the variables, so
evaluation at $u=\alpha$ commutes with $\varphi_{e}$ only when
$\varphi(\alpha)=\alpha$. Otherwise specialization can take a solution of the
Penrose system to a matrix that is not one, and \cite[Ex.~14]{SendraSendra2017}
gives a matrix over $\C(x_{1},x_{2},x_{3})$ that specializes correctly at every
real point and fails at $(1,1+i,1)$. Over $\R$ with the identity involution
$K_{\varphi}=\R$ and the restriction is empty.

As defined above, $\mathcal{E}(A)$ cannot in general be determined before $\pinv{A}$
has been computed. A computable superset can be extracted from the proof of
\cite[Thm.~12]{SendraSendra2017}; Corollary~\ref{thm:excep} below identifies the
exceptional set itself at the admissible parameter values, while in the application
of Section~\ref{sec:application} the relevant locus reduces to
$\{\det(I-A(u))=0\}$, which can be computed directly.

\begin{example}[the exceptional set and what happens on it]
\label{ex:excep}
Let the ground field be $\R(u)$ with the identity involution, a Moore--Penrose field
by Corollary~\ref{cor:Ruexists}, so that all real parameter values are admissible.
Consider
\[
  B=\begin{pmatrix}1&1\\1&u\end{pmatrix},
  \qquad \det B=u-1,
  \qquad
  \pinv{B}=B^{-1}=\frac{1}{u-1}\begin{pmatrix}u&-1\\-1&1\end{pmatrix},
\]
the second equality because $\rank B=2$ over $\R(u)$, where $\det B$ is a nonzero
polynomial. Hence $\den(B)=1$ and $\den(\pinv{B})=u-1$, so
$\mathcal{E}(B)=\{1\}$ and Theorem~\ref{thm:specialization} applies at every
$\alpha\neq1$.

At $\alpha=1$ the rank drops to one and
$\pinv{(B|_{u=1})}=\tfrac14\left(\begin{smallmatrix}1&1\\1&1\end{smallmatrix}
\right)$ still exists. It is not, however, obtained by specializing $\pinv{B}$,
since the closed form is undefined at $u=1$. The exceptional set should therefore
not be read as the locus where the Moore--Penrose inverse fails to exist, but as the
locus where the parametric closed form ceases to represent the Moore--Penrose
inverse of the specialized matrix. In Section~\ref{sec:application}, identifying
precisely this locus is part of the problem rather than something to be avoided.
\end{example}

\subsection{Structural reduction and the dimension of $A\{1,3,4\}$}
\label{ssec:reduction}

The four Penrose conditions do not all have the same degree in the unknowns. Take
the involution on $K$ to be the identity, so that $A^{*}=A^{\mathsf T}$. Every
entry of $AXA-A$, of $(AX)^{\mathsf T}-AX$ and of $(XA)^{\mathsf T}-XA$ is a sum
of terms with a single unknown factor, so \eqref{eq:P1}, \eqref{eq:P3} and
\eqref{eq:P4} form an affine linear system. The $(i,j)$ entry of $XAX$, on the
other hand, is $\sum_{k,l}x_{ik}a_{kl}x_{lj}$, and its monomials $x_{ik}x_{lj}$
are pairwise distinct, so the quadratic part survives whenever $A\neq0$. So,
unless $A=0$, only \eqref{eq:P2} is genuinely quadratic.

If the involution is not the identity, the situation changes. The entries of
$(AX)^{*}-AX$ and of $(XA)^{*}-XA$ then involve $\varphi(x_{ij})$ and not $x_{ij}$,
so the maps $X\mapsto(AX)^{*}$ and $X\mapsto(XA)^{*}$ are only
$\varphi$-semilinear: they are additive, but a scalar comes out conjugated,
$\varphi(\lambda)$ instead of $\lambda$. The system is therefore not $K$-linear, and
this is why we treat the complex case separately in
Section~\ref{ssec:complex}, where we replace it by a system that is linear over the
fixed field. Even so, the solution set is an affine $K$-subspace, a coset of a
$K$-subspace, because the argument of Theorem~\ref{prop:dimVL} uses only additivity
and not $K$-linearity.

We write $\mathcal{L}$ for the set of polynomials given by \eqref{eq:P1},
\eqref{eq:P3} and \eqref{eq:P4}, and $V(\mathcal{L})$ for its solution set, so that
$V(\mathcal{L})=A\{1,3,4\}$. Since $\pinv{A}$ satisfies the four conditions, it
belongs to $V(\mathcal{L})$, and it is the only point of $V(\mathcal{L})$ that
satisfies \eqref{eq:P2}. Any treatment of \eqref{eq:P2} therefore involves
$\dim V(\mathcal{L})$ unknowns and not $nm$, and the next theorem computes that
dimension.

The second criterion of Theorem~\ref{thm:mpfield} says precisely that the form
$\langle v,w\rangle=\sum_{k}\varphi(v_{k})w_{k}$ on $K^{s}$ is
\emph{anisotropic}, that is, that it vanishes on no nonzero vector, for every $s$.
This is a property of the field alone, and we use that name from here on.

\begin{lemma}[anisotropy]
\label{lem:anisotropy}
Let $(K,\varphi)$ be a Moore--Penrose field. Then $v^{*}v=0$ implies $v=0$ for
every $v\in K^{s}$ and every $s\ge1$, and consequently, for every
$A\in K^{m\times n}$,
\[
  \Ima(A)\cap\Ima(A)^{\perp}=\{0\}
  \qquad\text{and}\qquad
  \Ker(A)\cap\Ima(A^{*})=\{0\},
\]
where $U^{\perp}=\{y: y^{*}u=0 \text{ for all } u\in U\}$.
\end{lemma}

\begin{proof}
Writing $v=(v_{1},\dots,v_{s})^{\mathsf T}$ we have
$v^{*}v=\sum_{k}\varphi(v_{k})v_{k}$, so $v^{*}v=0$ forces $v=0$ by
Theorem~\ref{thm:mpfield}. If $y\in\Ima(A)\cap\Ima(A)^{\perp}$, then
$y^{*}u=0$ for every $u\in\Ima(A)$, by the definition of $\Ima(A)^{\perp}$. Since
$y$ belongs to $\Ima(A)$ as well, we may take $u=y$, which gives $y^{*}y=0$ and
therefore $y=0$. If $y\in\Ker(A)\cap\Ima(A^{*})$, say $y=A^{*}z$ with $Ay=0$, then
$y^{*}y=(A^{*}z)^{*}y=z^{*}Ay=0$, so $y=0$.
\end{proof}

\begin{example}[an isotropic vector destroys $A\{1,3,4\}$]
\label{ex:isotropic}
Take $K=\Q(i)$ with the \emph{identity} involution, so that $A^{*}=A^{\mathsf T}$,
and let $v=(1,i)^{\mathsf T}$. Then $v^{*}v=1+i^{2}=0$ with $v\neq0$, so $(K,\mathrm{id})$
is not a Moore--Penrose field. Put
\[
  A=v\,v^{\mathsf T}=\begin{pmatrix}1&i\\i&-1\end{pmatrix},
  \qquad \rank A=1 .
\]
The inner condition alone is satisfiable: $AXA=A$ holds on the three-parameter
family
\[
  X=\begin{pmatrix}
    1-i(x_{12}+x_{21})+x_{22} & x_{12}\\[2pt] x_{21} & x_{22}
  \end{pmatrix},
  \qquad x_{12},x_{21},x_{22}\in K,
\]
so $A\{1\}$ has dimension three, the classical value $nm-r^{2}$ of
\cite[Thm.~4.1]{StanimirovicLMA2021}. Adjoining \eqref{eq:P3} and \eqref{eq:P4}, however, leaves
nothing: $A\{1,3,4\}=\emptyset$, and a fortiori $\pinv{A}$ does not exist. The
rank criterion of Theorem~\ref{thm:mpfield} says the same in one line, since
$A^{2}=0$ here and so $\rank(AA^{*})=0\neq1=\rank(A)$. Theorem~\ref{prop:dimVL} is therefore not a statement
that can be made about $A\{1,3,4\}$ in general and then specialized; over an
isotropic field the set it describes can be empty while $A\{1\}$ is as large as
ever.

Over $\C$ with conjugation, by contrast,
$v^{*}v=1\cdot\overline{1}+i\cdot\overline{i}=2\neq0$, the field is Moore--Penrose
and $\pinv{A}$ exists: what fails over $(\Q(i),\mathrm{id})$ is the choice of
involution and not the arithmetic.
\end{example}

This also answers an obvious question: why not work simply over $\C$, where
existence is automatic. The results below would read the same over
$\C(u_{1},\dots,u_{\ell})$, so the generality needs a reason, and the reason is
computational. We compute over $\Q$ or over $\Q(i)$, and not over $\C$, whose
elements are not representable. Besides, the statements themselves depend on
which field we fix: the admissible parameter values of
Corollary~\ref{thm:excep} are those of $K_{\varphi}$, which is $\R$ for $\C$ with
conjugation and $\Q$ for $\Q$ with the identity.
\begin{theorem}[dimension of $A\{1,3,4\}$; classical over $\C$,
see Remark~\ref{rem:dimattribution}]
\label{prop:dimVL}
Let $(K,\varphi)$ be a Moore--Penrose field and $A\in K^{m\times n}$ with
$\rank(A)=r$. Let $N\in K^{n\times(n-r)}$ and $M\in K^{m\times(m-r)}$ have
columns forming bases of $\Ker(A)$ and of $\Ker(A^{*})$ respectively. Then
\begin{equation}
\label{eq:VLparam}
  A\{1,3,4\}=\bigl\{\,\pinv{A}+N\,E\,M^{*}\ :\ E\in K^{(n-r)\times(m-r)}\,\bigr\},
\end{equation}
and the map $E\mapsto \pinv{A}+NEM^{*}$ is a bijection. In particular
\[
  d \;:=\; \dim A\{1,3,4\} \;=\; (n-r)(m-r),
\]
a quantity we refer to as $d$ throughout. The statement holds for every
$r\le\min(m,n)$, the case $r=\min(m,n)$ giving $d=0$ and
$A\{1,3,4\}=\{\pinv{A}\}$.
\end{theorem}

\begin{proof}
As the solution set of an affine-linear system containing $\pinv{A}$,
$A\{1,3,4\}$ is the coset $\pinv{A}+\mathcal{Z}$, where $\mathcal{Z}$ is the
solution space of
\begin{equation}
\label{eq:homog}
  AZA=0,\qquad (AZ)^{*}=AZ,\qquad (ZA)^{*}=ZA.
\end{equation}
We claim $\mathcal{Z}=\{Z: AZ=0 \text{ and } ZA=0\}$. One inclusion is
immediate. For the converse let $Z$ satisfy \eqref{eq:homog} and put $W=AZ$, so
that $W$ is self-adjoint and $WA=AZA=0$. Fix $j$ and let $w=W_{\bullet j}$. From
$W=AZ$ we get $w=A\,Z_{\bullet j}\in\Ima(A)$; from $WA=0$ the product of the $j$-th row of
$W$ with $A$ is zero, and since $W$ is self-adjoint that row is $w^{*}$, so
$w^{*}A=0$ and $w\in\Ima(A)^{\perp}$. By Lemma~\ref{lem:anisotropy}, $w=0$, and as $j$ was
arbitrary $AZ=0$. The same argument applied to $V=ZA$ gives $ZA=0$: here $V$ is
self-adjoint, $AV=AZA=0$, the $j$-th column $v$ of $V$ lies in $\Ker(A)$, and
self-adjointness with $V=ZA$ gives
$v=A^{*}(Z_{j\bullet})^{*}\in\Ima(A^{*})$, whence $v=0$ by the second part of
the lemma.

It remains to identify $\mathcal{Z}$. The condition $AZ=0$ says that every column
of $Z$ lies in $\Ker(A)$, so $Z=NC$ for a unique $C\in K^{(n-r)\times m}$. The
condition $ZA=0$ says that every row $z$ of $Z$ satisfies $zA=0$, that is
$z^{*}\in\Ker(A^{*})$, so the rows of $Z$ lie in the span of the rows of $M^{*}$.
As $N$ has full column rank, $Z=NC$ has the same row space as $C$, so $C=EM^{*}$
for a unique $E$. Hence $Z=NEM^{*}$; conversely every such matrix satisfies
$AZ=0$ and $ZA=0$ because $AN=0$ and $M^{*}A=0$, and uniqueness of $E$ follows
from $N$ and $M$ having full column rank. Finally $(n-r)(m-r)=0$ if and only if
$r=n$ or $r=m$, that is $r=\min(m,n)$, and an affine subspace of dimension zero
containing $\pinv{A}$ is $\{\pinv{A}\}$.
\end{proof}

\begin{remark}[attribution]
\label{rem:dimattribution}
Theorem~\ref{prop:dimVL} is not new over $\C$. The characterizations of
$A\{1,3\}$ and $A\{1,4\}$ are classical \cite[Ch.~2]{BenIsraelGreville2003}, and
intersecting them gives
$A\{1,3,4\}=\pinv{A}+(I-\pinv{A}A)\,\C^{n\times m}\,(I-A\pinv{A})$, where
$I-\pinv{A}A$ and $I-A\pinv{A}$ are the orthogonal projectors onto $\Ker(A)$ and
$\Ima(A)^{\perp}$; that is \eqref{eq:VLparam} written with projectors in place of
kernel bases. We state the theorem because we require it over an arbitrary
Moore--Penrose field, where the form $\langle v,w\rangle$ is not positive definite,
so that what the classical argument takes for granted, namely
$\Ima(A)\cap\Ima(A)^{\perp}=\{0\}$, is not automatic. Anisotropy is what supplies
it, and Example~\ref{ex:isotropic} shows that it can fail: the proof above
accordingly goes through Lemma~\ref{lem:anisotropy} and uses no inner-product
structure beyond it.
\end{remark}

It is worth seeing how large this reduction is. The naive formulation has $nm$
unknowns, while the residual system has only $d=(n-r)(m-r)$ of them. For the
rank-two matrices of order $m\times(m+1)$ that we use in
Section~\ref{sec:cost} this means $m(m+1)$ unknowns against $(m-2)(m-1)$, that
is, seventy-two against forty-two when $m=8$. The next example, the smallest one
with $d>0$, is the one we use in the rest of the section.

\begin{example}[a rank-deficient parametric matrix, $d=1$]
\label{ex:running}
As in Example~\ref{ex:excep} let $K=\R(u)$ with the identity involution, and
\[
  A=\begin{pmatrix}1&u\\u&u^{2}\end{pmatrix}=v\,v^{\mathsf T},
  \qquad v=\begin{pmatrix}1\\u\end{pmatrix},
  \qquad \rank A=1 .
\]
Anisotropy is visible at once: $v^{*}v=1+u^{2}\neq0$ in $\R(u)$, and it is this
non-vanishing that makes the pseudoinverse
\[
  \pinv{A}=\frac{1}{(1+u^{2})^{2}}\begin{pmatrix}1&u\\u&u^{2}\end{pmatrix}
\]
available; the four conditions \eqref{eq:P1}--\eqref{eq:P4} are checked directly.
Here $m=n=2$ and $r=1$, so Theorem~\ref{prop:dimVL} gives $d=1$ against $nm=4$.
Since $A$ is symmetric, $\Ker(A)=\Ker(A^{*})$ is spanned by
$N=M=(-u,1)^{\mathsf T}$, and \eqref{eq:VLparam} reads
\begin{equation}
\label{eq:runningparam}
  A\{1,3,4\}
  =\Bigl\{\,\pinv{A}+t\begin{pmatrix}u^{2}&-u\\-u&1\end{pmatrix}
    \ :\ t\in\R(u)\,\Bigr\},
\end{equation}
a line in $\R(u)^{2\times2}$ through $\pinv{A}$, whose single unknown $t$ is the
entry of the $1\times1$ matrix $E$.
\end{example}

\subsection{The outer condition is affine-linear on $A\{1,3,4\}$}
\label{ssec:linearity}

As the degree count of Section~\ref{ssec:reduction} shows, \eqref{eq:P2} is the only
Penrose condition that is genuinely quadratic. Once restricted to $A\{1,3,4\}$, however, its quadratic
terms cancel identically and the residual becomes affine-linear. By
Theorem~\ref{prop:dimVL} only the case $d>0$ needs consideration: at full rank
$d=0$ and $A\{1,3,4\}=\{\pinv{A}\}$, so \eqref{eq:P2} is redundant.

The identity that we use is classical over $\C$: if $Y\in A\{1,4\}$ and
$Z\in A\{1,3\}$, then $YAZ=\pinv{A}$; see \cite{BenIsraelGreville2003}. Its proof
is short and it works over any Moore--Penrose field, so we recall it here. The
parametrizations
\[
  A\{1,3\}=\pinv{A}+(I-\pinv{A}A)K^{n\times m},
  \qquad
  A\{1,4\}=\pinv{A}+K^{n\times m}(I-A\pinv{A})
\]
follow from \eqref{eq:P1}--\eqref{eq:P4} and Lemma~\ref{lem:anisotropy}, which is
what makes the self-adjoint idempotent with image $\Ima(A)$ unique. The
computation rests on a single observation: by \eqref{eq:P1} we have
\[
  A\bigl(I-\pinv{A}A\bigr)=0
  \qquad\text{and}\qquad
  \bigl(I-A\pinv{A}\bigr)A=0 ,
\]
so what we add to $\pinv{A}$ is invisible to $A$, on one side or on the other.
Every cross term of $YAZ$ contains one of these two products, so it vanishes, and
only $\pinv{A}A\pinv{A}=\pinv{A}$ remains. Since
$A\{1,3,4\}=A\{1,3\}\cap A\{1,4\}$, we can take $Y=Z=X$, and we get
$XAX=\pinv{A}$ for every $X\in A\{1,3,4\}$.

\begin{theorem}[linearity of the residual system]
\label{thm:residuallinear}
Let $(K,\varphi)$ be a Moore--Penrose field and $A\in K^{m\times n}$. Then
\begin{equation}
\label{eq:residualidentity}
  XAX-X=\pinv{A}-X
  \qquad\text{for every }X\in A\{1,3,4\},
\end{equation}
and therefore the restriction of \eqref{eq:P2} to $A\{1,3,4\}$ is an affine-linear
system with unique solution $X=\pinv{A}$.
\end{theorem}

\begin{proof}
The identity above gives $XAX=\pinv{A}$ for every $X\in A\{1,3,4\}$. So on that
affine subspace the quadratic expression $XAX$ is not just of lower degree, it is
constant, and its quadratic terms cancel identically. Subtracting $X$ from both
sides gives \eqref{eq:residualidentity}, so the residual of \eqref{eq:P2} on
$A\{1,3,4\}$ is the displacement $\pinv{A}-X$ from the solution. In coordinates,
if we write $X=\pinv{A}+NEM^{*}$ as in \eqref{eq:VLparam}, that residual is
$-NEM^{*}$, which is linear in the $d$ entries of $E$ and vanishes only at $E=0$.
\end{proof}

The significance of Theorem~\ref{thm:residuallinear} is not the identity
$XAX=\pinv{A}$ but its residual form: the nonlinear Penrose condition becomes the
displacement from the unique solution. That is what
Remark~\ref{cor:twolinearsolves} and Corollary~\ref{rem:zerodim} use, and neither
follows from $\pinv{A}=YAZ$ without it.

\begin{remark}[$\pinv{A}$ by two linear solves]
\label{cor:twolinearsolves}
Over a Moore--Penrose field, $\pinv{A}$ can be computed in two stages:
\begin{enumerate}\itemsep2pt
  \item solve \eqref{eq:P1}, \eqref{eq:P3} and \eqref{eq:P4} to obtain the general
        element of $A\{1,3,4\}$;
  \item impose the residual equation $XAX-X=0$ on that affine space.
\end{enumerate}
When $\varphi=\mathrm{id}$ both stages are affine-linear over $K$, first in $nm$
unknowns and then in $d=(n-r)(m-r)$. When $\varphi\neq\mathrm{id}$ they are
affine-linear over the fixed field $K_{\varphi}$ instead, and the two counts
double, to $2nm$ and then $2d$, because $\{\mathrm{id},\varphi\}$ is a group of
order two and so $[K:K_{\varphi}]=2$. For $(\C,\overline{\phantom{a}})$ the two
unknowns of each entry are its real and its imaginary part.

For $\varphi=\mathrm{id}$ the first system is affine-linear over $K$ by the degree
count of Section~\ref{ssec:reduction}. For a nontrivial involution the entries of
$(AX)^{*}-AX$ and $(XA)^{*}-XA$ are $\varphi$-semilinear in the unknowns, so they
are additive and $K_{\varphi}$-linear but not $K$-linear. The solution set is
still the coset $\pinv{A}+\{NEM^{*}\}$, of dimension $d$ over $K$, because
additivity is all that the argument of Theorem~\ref{prop:dimVL} uses. Restricted
to it the second system is affine-linear with unique solution $\pinv{A}$ by
Theorem~\ref{thm:residuallinear}.
\end{remark}

Throughout what follows $\widetilde{X}$ denotes the matrix of unknowns after the
first linear solve, that is, the general element of $A\{1,3,4\}$ written with its
$d$ surviving indeterminates $x_{1},\dots,x_{d}$; the \emph{residual ideal} is the
ideal generated by the entries of $\widetilde{X}A\widetilde{X}-\widetilde{X}$.

\begin{corollary}[structure of the residual ideal]
\label{rem:zerodim}
The residual ideal is a maximal ideal of the polynomial ring, and in particular it
is radical and zero-dimensional over any extension of $K$.
\end{corollary}

\begin{proof}
By Theorem~\ref{thm:residuallinear} the entries of
$\widetilde{X}A\widetilde{X}-\widetilde{X}$ are $nm$ affine-linear forms in the
$d$ unknowns $x_{1},\dots,x_{d}$, of which exactly $d$ are independent, and they
vanish only at $\pinv{A}$. So Gaussian
elimination reduces the generating set to $\{x_{i}-c_{i}\}$, where
$c_{1},\dots,c_{d}\in K$ are the values of the unknowns for which
$\widetilde{X}=\pinv{A}$. This set generates the maximal ideal of the point
$(c_{1},\dots,c_{d})$.
\end{proof}

\begin{remark}[the reduced Gr\"obner basis of the residual ideal]
\label{rem:gb}
The set $\{x_{1}-c_{1},\dots,x_{d}-c_{d}\}$ of the proof above is already reduced
in any monomial order, so it is the reduced Gr\"obner basis of the residual
ideal, and it is the same one in every monomial order.
\end{remark}

Since $\pinv{A}$ is unique, it is natural to infer that the residual ideal is
zero-dimensional, and that is what is done, in one line, in
\cite[\S4]{SendraSendra2017}. Corollary~\ref{rem:zerodim} proves it, and it says
something more. Uniqueness of $\pinv{A}$ is a statement about the solutions with
entries in $K$, while the dimension of the ideal is a statement about the
solution set over the algebraic closure of $K$, and that is where the corollary
adds information: there the residual variety is still one single reduced point.

\begin{example}[Example~\ref{ex:running} continued: the residual system]
\label{ex:runningcont}
Substituting the parametrization \eqref{eq:runningparam} into the outer condition
gives
\[
  XAX-X
  \;=\;-\,t\begin{pmatrix}u^{2}&-u\\-u&1\end{pmatrix}
  \;=\;\pinv{A}-X ,
\]
which is \eqref{eq:residualidentity} in this instance: the quadratic terms in $t$
have cancelled. Equating entries to zero gives the residual system
\[
  -u^{2}t=0,\qquad ut=0,\qquad -t=0,
\]
three affine-linear equations in one unknown generating the ideal
$\langle t\rangle\subset\R(u)[t]$. This is the maximal ideal of the point $t=0$,
its reduced Gr\"obner basis is $\{t\}$, and back-substitution returns
$X=\pinv{A}$. The total degree of the residual
system in the unknowns is recorded for every matrix in the experiments of
Section~\ref{sec:cost} and is $1$ throughout, so
Theorem~\ref{thm:residuallinear} is tested on each instance rather than assumed.
\end{example}

\begin{remark}[the linearity belongs to the system, not to \eqref{eq:P2}]
\label{rem:notP2alone}
Theorem~\ref{thm:residuallinear} does not say that the outer condition is linear.
On its own, $XAX=X$ defines a set that is not an affine subspace, and this is
known: for $A\neq0$ over any field not isomorphic to $\mathbb{F}_{2}$,
Stanimirovi\'c et al.\ \cite[Lem.~4.1]{StanimirovicLMA2021} prove that $A\{2\}$ is
an affine algebraic variety and not an affine subspace, and both hypotheses are
automatic here. The one-line instance we need is that $X=0$ and $X=\pinv{A}$ both
satisfy $XAX=X$, while
$(t\pinv{A})A(t\pinv{A})-t\pinv{A}=(t^{2}-t)\pinv{A}$ vanishes only for
$t\in\{0,1\}$, so the segment joining two solutions consists of solutions only at
its endpoints.

For $A\in K^{1\times n}$ the whole set is visible at once: $XAX-X=(AX-1)X$ entry
by entry, so
\[
  A\{2\}=\{0\}\ \cup\ \{X\in K^{n\times1}: AX=1\},
\]
a hyperplane of dimension $n-1$ together with a point that does not lie on it,
worked out for $n=2$ in \cite[Ex.~4.2]{StanimirovicLMA2021}. That extra point is
exactly what obstructs affineness, and it is what
Theorem~\ref{thm:residuallinear} has to work around.

So Theorem~\ref{thm:residuallinear} is a statement about the Moore--Penrose
problem, and only once \eqref{eq:P1}, \eqref{eq:P3} and \eqref{eq:P4} have been
imposed. We claim nothing similar for $A\{2\}$ alone, nor for the other
$\{i,j\}$-inverses, and for those the Gr\"obner approach keeps its role.
\end{remark}

\subsection{The complex case}
\label{ssec:complex}

For complex matrices the hermitian conditions \eqref{eq:P3} and \eqref{eq:P4} are
not $\C$-linear in the entries of the unknown: the map $X\mapsto X^{*}$ conjugates
them, so $\overline{x_{ij}}$ is not a polynomial function of $x_{ij}$ and the
Penrose equations are not a polynomial system over $\C[x_{11},\dots,x_{nm}]$.
Following the real--imaginary splitting proposed in \cite[\S4]{SendraSendra2017},
write $A=A_{0}+iB_{0}$ and $X=X_{0}+iY_{0}$ with $A_{0},B_{0}\in\R^{m\times n}$ and
$X_{0},Y_{0}\in\R^{n\times m}$; this replaces the $nm$ complex unknowns by $2nm$
real ones. The following theorem carries the splitting out, giving the four real
matrix equations explicitly.

\begin{theorem}[real splitting of the hermitian conditions; M.\ Coronado]
\label{thm:complex}
Let $A=A_{0}+iB_{0}$ with $A_{0},B_{0}\in\R^{m\times n}$ and
$X=X_{0}+iY_{0}$ with $X_{0},Y_{0}\in\R^{n\times m}$. Then \eqref{eq:P3} holds
if and only if
\begin{align}
  \bigl(A_{0}X_{0}-B_{0}Y_{0}\bigr)^{\mathsf T}
    -\bigl(A_{0}X_{0}-B_{0}Y_{0}\bigr) &= 0, \tag{C1}\label{eq:C1}\\
  \bigl(A_{0}Y_{0}+B_{0}X_{0}\bigr)^{\mathsf T}
    +\bigl(A_{0}Y_{0}+B_{0}X_{0}\bigr) &= 0, \tag{C2}\label{eq:C2}
\end{align}
and \eqref{eq:P4} holds if and only if
\begin{align}
  \bigl(X_{0}A_{0}-Y_{0}B_{0}\bigr)^{\mathsf T}
    -\bigl(X_{0}A_{0}-Y_{0}B_{0}\bigr) &= 0, \tag{C3}\label{eq:C3}\\
  \bigl(Y_{0}A_{0}+X_{0}B_{0}\bigr)^{\mathsf T}
    +\bigl(Y_{0}A_{0}+X_{0}B_{0}\bigr) &= 0. \tag{C4}\label{eq:C4}
\end{align}
Equivalently: $AX$ is hermitian if and only if $\operatorname{Re}(AX)$ is
symmetric and $\operatorname{Im}(AX)$ is skew-symmetric, and likewise for $XA$.
\end{theorem}

\begin{proof}
Write $AX=U+iV$, where $U=A_{0}X_{0}-B_{0}Y_{0}$ and $V=A_{0}Y_{0}+B_{0}X_{0}$ are
both real. Then $(AX)^{*}=U^{\mathsf T}-iV^{\mathsf T}$, so $(AX)^{*}=AX$ is
equivalent to $U^{\mathsf T}=U$ and $V^{\mathsf T}=-V$, which is
\eqref{eq:C1}--\eqref{eq:C2}. The same argument applied to $XA$ gives
\eqref{eq:C3}--\eqref{eq:C4}.
\end{proof}

The complex Penrose system can therefore be treated as a real affine-linear system
in the real and imaginary parts of the unknown entries, which is what
Algorithm~\ref{alg:complex} of Section~\ref{sec:algorithms} does.

\section{Degrees and the exceptional set}
\label{sec:degrees}

The previous section leaves two questions open. The first one is a question of
cost. Theorem~\ref{thm:residuallinear} reduces the computation to linear algebra
over $K$, but an operation in $K=\Q(u_{1},\dots,u_{\ell})$ is an operation with
rational functions, and in practice it involves simplification, gcd computation
and normalization. The cost of all that grows with the degrees of the operands.
So, if we want a cost model that means something, we have to bound those degrees.
The second one is a question of structure. The exceptional set $\mathcal{E}(A)$
of Definition~\ref{def:exceptional} is defined in terms of
$\operatorname{den}(\pinv{A})$, so we cannot know it until we have computed
$\pinv{A}$. This is not satisfactory when the exceptional set is precisely what
we are interested in, as happens in Section~\ref{sec:application}.

We answer both questions with a classical determinantal representation, that we
recall first and use afterwards. Let $A\in K^{m\times n}$ have rank $r$, and let
us write
\begin{equation}
\label{eq:Sdef}
  S(A)\;:=\;\sum_{|I|=|J|=r}\varphi\bigl(\det A[I,J]\bigr)\cdot\det A[I,J],
\end{equation}
where the sum runs over all $I\subseteq\{1,\dots,m\}$ and all
$J\subseteq\{1,\dots,n\}$ of cardinality $r$. When $\varphi$ is the identity, $S(A)$ is the sum of the squares of
all the maximal minors of $A$, that is, the squared volume of $A$ in the sense of
Ben-Israel \cite{BenIsrael1992}.

\begin{theorem}[{Bapat, Bhaskara Rao and Prasad
\cite[Thm.~13]{BapatBhaskaraPrasad1990}}]
\label{thm:bbp}
Let $R$ be an integral domain and let $A\in R^{m\times n}$ have rank $r$. For the
identity involution, $\pinv{A}$ exists over $R$ if and only if $S(A)$ is invertible
in $R$, and in that case
\begin{equation}
\label{eq:closedid}
  S(A)\cdot\bigl(\pinv{A}\bigr)_{ji}
  \;=\;\sum_{I\ni i}\ \sum_{J\ni j}
      \det A[I,J]\cdot\operatorname{cof}\bigl(A[I,J];\,i,\,j\bigr),
\end{equation}
where $\operatorname{cof}(A[I,J];i,j)$ is the cofactor of $a_{ij}$ inside
$A[I,J]$.
\end{theorem}

The involutory case is not covered there: the final remark of that paper states
that the results of its Section~4 carry over to an involution ``with the necessary
modifications'', without giving either the modifications or a proof. Since
Section~\ref{sec:degrees} and Corollary~\ref{thm:excep} rest on that
extension, Proposition~\ref{prop:closedform} below supplies it for Moore--Penrose
fields. The special case in which $A$ has full row or column rank is due to
Bruening \cite{Bruening1987}, and the corresponding statements over an arbitrary
field go back to Gabriel \cite{Gabriel1969}. Over $\R$ with the transpose,
\eqref{eq:closedid} is Ben-Israel's reading of $\pinv{A}$ as a convex combination
of the inverses of the maximal nonsingular submatrices of $A$, weighted by their
squared volumes \cite{BenIsrael1992}.

The proof of Proposition~\ref{prop:closedform} settles one more point, which is
easy to get wrong. In the closed form the involution has to be applied to the
minor, and not to the cofactor. The two
choices agree when $r=1$, and they already differ when $r=2$, so it is worth
saying which one of them is the correct one. We prove it with the adjugate identity of
Lemma~\ref{lem:adj} below, which comes in turn from Cramer's rule and the
Cauchy--Binet formula, and with one further appeal to Cauchy--Binet, which turns
$\det(F^{*}F)$ into $S(F)$.

\begin{lemma}[adjugate identity]
\label{lem:adj}
Let $F\in K^{m\times r}$, where $K$ is a field with involution. Then for all
$1\le p\le r$ and $1\le i\le m$,
\begin{equation}
\label{eq:adj}
  \sum_{|I|=r}\varphi\bigl(\det F[I,:]\bigr)\cdot
  \operatorname{cof}\bigl(F[I,:];\,i,\,p\bigr)
  \;=\;\bigl(\operatorname{adj}(F^{*}F)\,F^{*}\bigr)_{pi},
\end{equation}
where the terms with $i\notin I$ are read as zero.
\end{lemma}

\begin{proof}
Fix $p$ and $i$, let $v:=e_{i}\in K^{m}$, and for an $r$-subset $I$ let
$v_{I}\in K^{r}$ be the subvector of $v$ on the rows $I$. Then $v_{I}=0$ when
$i\notin I$, and otherwise $v_{I}$ is the $k$-th standard vector, $k$ the position
of $i$ in $I$, so that
$\operatorname{cof}(F[I,:];i,p)=\bigl(\operatorname{adj}(F[I,:])\,v_{I}\bigr)_{p}$.
Write $M^{(p\leftarrow w)}$ for $M$ with its $p$-th column replaced by $w$.
Cramer's rule in the form
$(\operatorname{adj}(M)\,w)_{p}=\det M^{(p\leftarrow w)}$ turns the left side of
\eqref{eq:adj} into
\[
  \sum_{|I|=r}\varphi\bigl(\det F[I,:]\bigr)\,
  \det\bigl(F[I,:]^{(p\leftarrow v_{I})}\bigr)
  \;=\;\sum_{|I|=r}\varphi\bigl(\det F[I,:]\bigr)\,
  \det\bigl(F^{(p\leftarrow v)}[I,:]\bigr),
\]
deleting rows and replacing a column being operations that commute. Since
$\det\bigl(F^{*}[:,I]\bigr)=\varphi\bigl(\det F[I,:]\bigr)$, that sum is the
Cauchy--Binet expansion of $\det\bigl(F^{*}F^{(p\leftarrow v)}\bigr)$ along the
$r$-subsets of rows.
Finally the $q$-th column of $F^{*}F$ is $F^{*}f_{q}$, whence
$F^{*}F^{(p\leftarrow v)}=(F^{*}F)^{(p\leftarrow F^{*}v)}$, and Cramer's rule
again gives
$\det\bigl((F^{*}F)^{(p\leftarrow F^{*}v)}\bigr)
=\bigl(\operatorname{adj}(F^{*}F)F^{*}v\bigr)_{p}$, which is the right side.
\end{proof}

Nothing in Lemma~\ref{lem:adj} requires $F^{*}F$ to be invertible, and nothing in
it requires the field to be a Moore--Penrose field. Both enter at the next step,
and only there.

\begin{proposition}
\label{prop:closedform}
Let $(K,\varphi)$ be a Moore--Penrose field and let $A\in K^{m\times n}$ have rank
$r$. Then $S(A)\neq0$ and
\begin{equation}
\label{eq:closedentry}
  S(A)\cdot\bigl(\pinv{A}\bigr)_{ji}
  \;=\;\sum_{I\ni i}\ \sum_{J\ni j}
      \varphi\bigl(\det A[I,J]\bigr)\cdot
      \operatorname{cof}\bigl(A[I,J];\,i,\,j\bigr),
\end{equation}
which is \eqref{eq:closedid} when $\varphi$ is the identity.
\end{proposition}

\begin{proof}
Let first $A=F$ have full column rank $r$, so that $J=\{1,\dots,r\}$ is the only
admissible column set. By Cauchy--Binet as above,
$\det(F^{*}F)=\sum_{|I|=r}\varphi(\det F[I,:])\det F[I,:]=S(F)$, which is nonzero
by Lemma~\ref{lem:anisotropy} applied to the vector of maximal minors of $F$;
this is the one appeal to anisotropy in the argument. Hence $F^{*}F$ is invertible
and a direct check of \eqref{eq:P1}--\eqref{eq:P4} gives $\pinv{F}=(F^{*}F)^{-1}F^{*}$, so
that $\operatorname{adj}(F^{*}F)F^{*}=S(F)\pinv{F}$ and \eqref{eq:adj} is
\eqref{eq:closedentry}. Taking adjoints, the same identity holds when $A$ has full
row rank.

For general $r$ write $A=BC$ with $B\in K^{m\times r}$ of full column rank and
$C\in K^{r\times n}$ of full row rank. As $B^{*}B$ and $CC^{*}$ are invertible by
the previous paragraph, one checks \eqref{eq:P1}--\eqref{eq:P4} for
$C^{*}(CC^{*})^{-1}(B^{*}B)^{-1}B^{*}$ directly, since $AX=B(B^{*}B)^{-1}B^{*}$ and
$XA=C^{*}(CC^{*})^{-1}C$ are self-adjoint by construction, and therefore
$\pinv{A}=\pinv{C}\,\pinv{B}$. Now $A[I,J]=B[I,:]\,C[:,J]$ is a product of two
$r\times r$ matrices, so
\begin{equation}
\label{eq:factorminor}
  \det A[I,J]=\det B[I,:]\cdot\det C[:,J],
  \qquad
  \operatorname{adj}A[I,J]=\operatorname{adj}C[:,J]\cdot\operatorname{adj}B[I,:],
\end{equation}
the second being the anti-multiplicativity of the adjugate, an identity of
polynomials in the entries and hence valid over any commutative ring. The first
identity in \eqref{eq:factorminor} makes $S$ multiplicative on the factorization,
$S(A)=S(B)S(C)$, since the double sum \eqref{eq:Sdef} separates. The second turns
the product of the two full-rank formulas into the double sum
\eqref{eq:closedentry}: reading off the $(j,i)$ entry of
$S(A)\pinv{A}=\bigl[S(C)\pinv{C}\bigr]\bigl[S(B)\pinv{B}\bigr]$ and expanding each
factor by the full-rank case gives
\[
  S(A)\cdot\bigl(\pinv{A}\bigr)_{ji}
  =\sum_{I\ni i}\sum_{J\ni j}
   \varphi\bigl(\det B[I,:]\bigr)\varphi\bigl(\det C[:,J]\bigr)
   \sum_{p=1}^{r}\operatorname{cof}\bigl(C[:,J];p,j\bigr)
                 \operatorname{cof}\bigl(B[I,:];i,p\bigr),
\]
and the inner sum is the $(j,i)$ entry of
$\operatorname{adj}C[:,J]\cdot\operatorname{adj}B[I,:]$, that is
$\operatorname{cof}(A[I,J];i,j)$.
\end{proof}

We record two features of \eqref{eq:closedentry} before we use it.

First, the criterion for $\pinv{A}$ to exist is our standing hypothesis, that
$(K,\varphi)$ is a Moore--Penrose field, written in a different way. Over a field
every nonzero element is invertible, so the condition in Theorem~\ref{thm:bbp} and
Proposition~\ref{prop:closedform} is just $S(A)\neq0$. And Lemma~\ref{lem:anisotropy}, applied to the vector of maximal
minors instead of to a column, gives that $S(A)\neq0$ for every $A$ of rank $r$ if
and only if $(K,\varphi)$ is a Moore--Penrose field: one implication is anisotropy
applied to the minors, and for the other one an isotropic vector
$(a_{1},\dots,a_{s})$ is a rank-one row matrix whose maximal minors are the
$a_{k}$, so $S$ vanishes on it, as Example~\ref{ex:isotropic} shows. So we have one
single condition here, and not three: the
one that makes $\pinv{A}$ exist, the one that fixes $\dim A\{1,3,4\}$ and the one
that makes the denominator in \eqref{eq:closedentry} nonzero are the same
condition.

Second, \eqref{eq:closedentry} gives us an \emph{intrinsic common denominator}
for $\pinv{A}$. Intrinsic means that $S(A)$ clears the denominators of all the
entries and depends only on the maximal minors of $A$, not on any elimination or
factorization along the way. It does not have to be the reduced denominator, as
Remark~\ref{rem:tight} shows, only a multiple of it. In
Section~\ref{sec:degrees} it becomes a degree bound, and in
Corollary~\ref{thm:excep} an exact description of the exceptional set.

From here on $(K,\varphi)$ is a Moore--Penrose field and
$P=K[u_{1},\dots,u_{\ell}]$. The parametric field is $\Kat$, with $\varphi$
extended to $\varphi_{e}$, and it is again a Moore--Penrose field by
Theorem~\ref{thm:mpfield}. We recall two things about this extension. First,
$\varphi_{e}$ acts on the coefficients and fixes the variables, so it preserves
degrees. Second, the admissible parameter values of
Theorem~\ref{thm:specialization} are the ones in $K_{\varphi}^{\,\ell}$, where
the fixed field is the one of the \emph{ground} field $K$, and not the one of
$\Kat$. We take $A\in P^{m\times n}$ with entries of total degree at most
$\delta$ and of rank $r\ge1$ over $\Kat$, and we exclude the trivial case $A=0$.
A reader who prefers a concrete setting can take $(K,\varphi)=(\R,\mathrm{id})$
throughout, with the real points as admissible values.

Asking the entries to be \emph{polynomial} is only a normalization, and we lose no
generality with it. Take any $A\in\Kat^{m\times n}$ and put
$q=\operatorname{den}(A)$ and $\widetilde{A}=qA$. The two matrices differ by a
nonzero scalar of $\Kat$, so they have the same rank, and the four conditions give
$\pinv{(cA)}=c^{-1}\pinv{A}$ for every nonzero $c\in\Kat$. Therefore
\[
  \pinv{A}=q\,\pinv{\widetilde{A}},
\]
and clearing denominators once transfers every statement below to $A$: if the
entries of $\widetilde{A}$ have degree at most $\widetilde{\delta}$, then
$\deg\num(\pinv{A})\le\deg q+(2r-1)\widetilde{\delta}$ and
$\deg\den(\pinv{A})\le2r\widetilde{\delta}$. From now on we work with
$\widetilde{A}$ and we drop the tilde.

Bounds of this kind are not new. One of them follows from Kalman's full-rank
factorization
$\pinv{A}=C^{*}(B^{*}AC^{*})^{-1}B^{*}$ \cite[Thm.~3]{Kalman1976}, which is
recalled as \cite[Thm.~2 and Rem.~1]{SendraSendra2017} and observed there,
although not stated, in \cite[\S4]{SendraSendra2017}. In total degree this route
gives $\deg\num\le(3r-1)\delta$ and $\deg\den\le3r\delta$.

There is a second route over $\R$ with the transpose, the Leverrier--Faddeev
algorithm of \cite{KarampetakisLAA1997}. Its denominator is $a_{k}$, the last
nonvanishing coefficient of the characteristic polynomial of $AA^{\mathsf T}$.
That index is the rank, and by Cauchy--Binet $a_{r}=S(A)$, so both denominators
are the same polynomial and a degree count gives $(2r-1)\delta$ and $2r\delta$
again. The bounds stated in \cite{VologiannidisKarampetakis2004} are that same
bound in its worst case, given per variable and with the number of rows in place
of the rank, because an interpolation algorithm does not know the rank in
advance.

So Theorem~\ref{thm:bound} does not contribute a sharper constant. What it
contributes is an explicit statement, in total degree and over an arbitrary
Moore--Penrose field, of a bound that the case of the identity involution leaves
implicit, and, through $S(A)$, the description of the exceptional set in
Corollary~\ref{thm:excep}. It does improve the two bounds above, the denominator
one by a factor $3/2$ and the numerator one by $(3r-1)/(2r-1)$, and
Remark~\ref{rem:tight} reports that the second one is attained in every instance
that we tested. But what matters here is not sharpness, it is canonicity:
$\det(B^{*}AC^{*})$ depends on the blocks that we choose and it can vanish at
points where the rank of $A$ does not drop, while $S(A)$ is determined by $A$ and
vanishes exactly on the rank-drop locus. It is the second one that
Corollary~\ref{thm:excep} needs.

\begin{theorem}[degree bound]
\label{thm:bound}
$S(A)\pinv{A}$ is a polynomial matrix, and
\[
  \deg S(A)\le 2r\delta,
  \qquad
  \deg\bigl(S(A)\pinv{A}\bigr)\le(2r-1)\delta .
\]
In particular $\operatorname{den}(\pinv{A})$ divides $S(A)$, so the numerator of
$\pinv{A}$ has degree at most $(2r-1)\delta$ and its denominator has degree at
most $2r\delta$.
\end{theorem}

\begin{proof}
We apply Proposition~\ref{prop:closedform} over the field $\Kat$ itself, where
$S(A)\neq0$ and therefore $S(A)$ is invertible. This is the ring that we have to
take, because over $P$ the polynomial $S(A)$ is not constant, so it is not
invertible either, and Theorem~\ref{thm:bbp} would give us nothing. Identity
\eqref{eq:closedentry} then holds with entries in $\Kat$, and its right-hand side
is polynomial in the entries of $A$. So $S(A)\pinv{A}$ is a polynomial matrix.

Now the degrees. Each $\det A[I,J]$ is a sum of products of $r$ entries of $A$, so
it has degree at most $r\delta$. And $\varphi_{e}$ fixes the variables, so
$\varphi_{e}(\det A[I,J])$ has the same degree. Then every summand of
\eqref{eq:Sdef} has degree at most $2r\delta$. In \eqref{eq:closedentry} each
summand is one of these minors times a cofactor of an $r\times r$ submatrix, and
such a cofactor has degree at most $(r-1)\delta$, so the summand has degree at
most $(2r-1)\delta$.
\end{proof}

Neither of the two bounds involves $m$ or $n$ explicitly, so $\delta$ and $r$
alone control the degree of each entry of $\pinv{A}$. The order of the matrix
enters only through $r\le\min(m,n)$ and through the number of entries.

\begin{corollary}[the exceptional set, computed from $A$]
\label{thm:excep}
Among the admissible parameter values, that is for
$\alpha\in K_{\varphi}^{\,\ell}$,
\begin{equation}
\label{eq:excep}
  \underbrace{\mathcal{E}(A)\cap K_{\varphi}^{\,\ell}}_{\text{specialization fails}}
  \;=\;\underbrace{\{\alpha\in K_{\varphi}^{\,\ell}:S(A)(\alpha)=0\}}_{\text{(a)}}
  \;=\;\underbrace{\{\alpha\in K_{\varphi}^{\,\ell}:\rank A|_{u=\alpha}<r\}}_{\text{(b)}}.
\end{equation}
\end{corollary}

So, inside $K_{\varphi}^{\,\ell}$, the exceptional set is the zero set of one
polynomial of degree at most $2r\delta$, by Theorem~\ref{thm:bound}, and that
polynomial is determined by $A$ alone.

The equality of (a) and (b) is elementary and we do not claim it. It is anisotropy
applied to the vector of maximal minors, and over $(\R,\mathrm{id})$ it says that
a sum of real squares vanishes only when all the summands vanish. In one variable
this identification is standard, and it is the one behind the analysis of the
poles of $\pinv{A}$ for rational transfer-function matrices in
\cite{ZhangPoles2015} and of the finite poles of $\pinv{A}$ for polynomial
matrices in \cite{KafetzisKarampetakis2022}.

What Corollary~\ref{thm:excep} adds is the first equality, and that one is a
statement about \emph{specialization}. It does not only say that $\pinv{A}$ has a
pole where the rank drops. It says that the parametric closed form stops
representing $\pinv{\bigl(A|_{u=\alpha}\bigr)}$ at exactly those points, with
$\ell$ parameters and over an arbitrary Moore--Penrose field. This is the
direction that needs the trace argument of the proof below.

\begin{proof}
Since $A$ is polynomial, $\operatorname{den}(A)=1$, and so the part of
$\mathcal{E}(A)$ inside $K_{\varphi}^{\,\ell}$ is the set of zeros of
$\operatorname{den}(\pinv{A})$ there.

We begin with (a) $=$ (b). The involution $\varphi_{e}$ fixes the variables and
$\alpha\in K_{\varphi}^{\,\ell}$, so evaluation commutes with the involution and
$S(A)(\alpha)=\sum\varphi(g_{IJ})\,g_{IJ}$, where
$g_{IJ}=\det A[I,J]|_{u=\alpha}\in K$. By Lemma~\ref{lem:anisotropy} this sum
vanishes if and only if all the $g_{IJ}$ vanish, that is, if and only if the rank
of $A|_{u=\alpha}$ drops below $r$.

We go now to the first equality. One inclusion is Theorem~\ref{thm:bound},
because $\operatorname{den}(\pinv{A})$ divides $S(A)$. For the other one, take
$\alpha\in K_{\varphi}^{\,\ell}$ with $\rank A|_{u=\alpha}=r'<r$ and suppose that
$\operatorname{den}(\pinv{A})(\alpha)\neq0$. Then \eqref{eq:speccond} holds, and
Theorem~\ref{thm:specialization} gives
$\bigl(\pinv{A}\bigr)\big|_{u=\alpha}=\pinv{\bigl(A|_{u=\alpha}\bigr)}$. On the
other hand, $\pinv{A}A$ is idempotent by \eqref{eq:P1} and \eqref{eq:P2} and it
has rank $r$. Over a field of characteristic zero the trace of an idempotent
matrix is its rank, and a Moore--Penrose field has characteristic zero by
\cite[Cor.~5]{SendraSendra2017}, so $\trace(\pinv{A}A)=r$ \emph{identically},
that is, it is a constant rational function. We specialize that constant at
$\alpha$ and we use the equality above,
\[
  r=\trace\bigl(\pinv{A}A\bigr)\big|_{u=\alpha}
   =\trace\Bigl(\pinv{\bigl(A|_{u=\alpha}\bigr)}\,A|_{u=\alpha}\Bigr)=r',
\]
where the last step holds because that matrix is again idempotent, of rank
$\rank A|_{u=\alpha}=r'$. This contradicts $r'<r$.
\end{proof}

\begin{remark}[what the equality does and does not say]
\label{rem:excepequal}
Corollary~\ref{thm:excep} is an equality of zero sets, and not an equality of
polynomials. In general $\operatorname{den}(\pinv{A})$ is a proper divisor of
$S(A)$, because common factors can cancel between numerator and denominator. What
the corollary says is that no cancellation of this kind can remove a zero that
lies in $K_{\varphi}^{\,\ell}$. This matters because
Definition~\ref{def:exceptional} describes $\mathcal{E}(A)$ through $\pinv{A}$,
while Corollary~\ref{thm:excep} determines it from $A$. Cancellation can change
multiplicities, but it cannot change the set of zeros in $K_{\varphi}^{\,\ell}$.
And restricting to the values $\alpha\in K_{\varphi}^{\,\ell}$ costs us nothing,
because these are the admissible ones, and it is what makes the statement
unconditional.
\end{remark}

\begin{example}[Example~\ref{ex:running} again]
\label{ex:runningS}
Take $(K,\varphi)=(\R,\mathrm{id})$. For
$A=\begin{pmatrix}1&u\\u&u^{2}\end{pmatrix}$, of rank one over $\R(u)$ with
$\delta=2$, the maximal minors are the entries, so $S(A)=(1+u^{2})^{2}$ and
\eqref{eq:closedentry} reduces to $\pinv{A}=A^{\mathsf T}/S(A)$, the matrix computed
there. Both bounds are attained: the denominator degree $4$ is $2r\delta$ and the
numerator degree $2$ is $(2r-1)\delta$. Since $1+u^{2}$ has no real zero,
$\mathcal{E}(A)=\emptyset$.
\end{example}

\begin{remark}[sharpness]
\label{rem:tight}
In thirty randomly generated instances, in one and in two parameters and with $r$
up to three, the degree of $S(A)\pinv{A}$ equalled $(2r-1)\delta$ in every case.
We do not prove that the bound is attained for all $(r,\delta,\ell)$. What these
instances show is that, on this family, it cannot be improved as a function of $r$
and $\delta$ alone.

For the denominator there is a criterion, and it is not rank deficiency. Write
$c(A)$ for the greatest common divisor of the maximal minors of $A$. By
\eqref{eq:closedentry} every summand of $S(A)\pinv{A}$ carries a factor
$\det A[I,J]$, so $c(A)$ divides the numerator. And by \eqref{eq:Sdef} the square
$c(A)^{2}$ divides $S(A)$. So one factor $c(A)$ cancels in the reduction to
lowest terms, and $\deg\operatorname{den}(\pinv{A})\le2r\delta-\deg c(A)$. So
coprimality of the maximal minors is \emph{necessary} for the bound to be
attained, and in every instance that we tested it was also sufficient. The failure
has nothing to do with the rank. Take for example
$A=\begin{pmatrix}1&u&u+1&2\\1&u&u+1&2\\u&1&0&u-1\end{pmatrix}$, of rank two with
$\delta=1$. Here $c(A)=u+1$, the polynomial $S(A)=8(u+1)^{2}(u^{2}-2u+2)$ has
degree $4=2r\delta$, and $\operatorname{den}(\pinv{A})=4(u+1)(u^{2}-2u+2)$ has
degree three. The same criterion accounts for the full-rank cells of
Table~\ref{tab:det}, as Section~\ref{ssec:det} notes. When $r=m=n$ the shortfall
is structural: there $\pinv{A}=A^{-1}$, so $\operatorname{den}(\pinv{A})$ divides
$\det A$ and has degree at most $r\delta$, while $S(A)=\det(A)\varphi(\det A)$ can
have degree $2r\delta$.
\end{remark}

Theorem~\ref{thm:bound} bounds the \emph{output}. The reduced entries of
$\pinv{A}$ have degrees that are linear in the entry degree and in the rank, and
that do not depend on the order of the matrix. It says nothing about the
intermediate expressions that Gaussian elimination and the calls to
\texttt{normal} or \texttt{simplify} produce along the way. Their growth is a
different phenomenon, and as Section~\ref{ssec:phases} shows it is the dominant
cost. So we have a bound on the size of the output, in agreement with
Remark~\ref{rem:params-are-coeffs}, and the sweeps over the rank and over the
number of parameters in Section~\ref{ssec:phases} test whether the measured cost
follows those same two parameters.

\section{Algorithms}
\label{sec:algorithms}

By Remark~\ref{cor:twolinearsolves} the three algorithms below are one single
method. We make two affine-linear solves: the first one imposes \eqref{eq:P1},
\eqref{eq:P3} and \eqref{eq:P4} on the $nm$ unknowns, and the second one imposes
\eqref{eq:P2} on the $d$ unknowns that survive. What changes from one to the other is the ring where the
unknowns live and, in the parametric case, which symbols are unknowns at all.

Their roles are different, though. Algorithm~\ref{alg:real} is the method in its
simplest form and we claim no novelty for it, since computing $\pinv{A}$ over $\Q$
is classical; we state it because Algorithm~\ref{alg:param} is a modification of
it and the correctness argument goes through it. Algorithm~\ref{alg:complex}
carries Theorem~\ref{thm:complex}, and without that theorem the complex case is
not a linear system at all; that theorem and the first implementations of the
three algorithms are due to M.\ Coronado, as the acknowledgements record.
Algorithm~\ref{alg:param} is the only one to which Theorem~\ref{thm:bound} and
Corollary~\ref{thm:excep} say anything, since with entries in $\Q$ there is no
parameter to bound or to specialize.

The inputs below are rational, because that is what a computer algebra system
represents exactly, and by Theorem~\ref{thm:mpfield} both $\Q$ and
$\Q(u_{1},\dots,u_{\ell})$ are Moore--Penrose fields for the identity involution,
so the theory of Sections~\ref{sec:theory} and~\ref{sec:degrees} applies to them
word for word. We leave out the fourth combination, rational functions over $\C$
with conjugation: the theory covers it, but there the admissible parameter values
are the ones in $K_{\varphi}^{\,\ell}=\R^{\ell}$, so the specialization results
only apply at real parameters.

\begin{algo}[H]
\caption{$\textsc{MoorePenroseR}(A)$: real entries}
\label{alg:real}
\KwIn{$A\in\Q^{m\times n}$, given exactly.}
\KwOut{$\pinv{A}\in\Q^{n\times m}$.}
\begin{enumerate}\itemsep2pt
  \item $X\leftarrow(x_{ij})$, a matrix of $nm$ indeterminates.
  \item $\mathcal{L}\leftarrow$ entries of $AXA-A$, of $(AX)^{\mathsf T}-AX$ and
        of $(XA)^{\mathsf T}-XA$.
        \hfill{\small[affine linear, \S\ref{ssec:reduction}]}
  \item Solve $\mathcal{L}=0$; substitute the affine solution into $X$ to obtain
        $\widetilde{X}$.
  \item \textbf{if} $\widetilde{X}$ contains no indeterminate
        \textbf{then return} $\widetilde{X}$.
        \hfill{\small[$d=0$; Thm.~\ref{prop:dimVL}]}
  \item $\mathcal{M}\leftarrow$ entries of
        $\widetilde{X}A\widetilde{X}-\widetilde{X}$, expanded.
        \hfill{\small[affine-linear, by Thm.~\ref{thm:residuallinear}]}
  \item Solve $\mathcal{M}=0$, substitute into $\widetilde{X}$ and
        \textbf{return} the result.
\end{enumerate}
\end{algo}

Step~5 needs a word. The entries of $\widetilde{X}A\widetilde{X}-\widetilde{X}$
are formed as products, so as they are written they are quadratic, and the
quadratic terms cancel only after we expand the expression. So an implementation
has to expand before solving, and it should not test the degree of the unexpanded
form. Both systems have coefficients in $\Q$ and we solve them by Gaussian
elimination over $\Q$, so the output is rational. And since the pseudoinverse of
a real matrix is real, we do not need complex unknowns here: $nm$ of them are
enough, instead of the $2nm$ of Algorithm~\ref{alg:complex}.

\begin{algo}[H]
\caption{$\textsc{MoorePenroseC}(A)$: complex entries}
\label{alg:complex}
\KwIn{$A=A_{0}+iB_{0}$ with $A_{0},B_{0}\in\Q^{m\times n}$.}
\KwOut{$\pinv{A}\in\Q(i)^{n\times m}$.}
\begin{enumerate}\itemsep2pt
  \item $X\leftarrow(x_{ij})$, $Y\leftarrow(y_{ij})$, each of $nm$
        indeterminates \emph{over $\Q$}; \quad $Z\leftarrow X+iY$.
  \item $A_{0}\leftarrow\operatorname{coeff}(A,i,0)$,
        \quad $B_{0}\leftarrow\operatorname{coeff}(A,i,1)$.
  \item $\mathcal{L}\leftarrow$ the real and imaginary parts of the entries of
        $AZA-A$, together with the entries of \eqref{eq:C1}--\eqref{eq:C4}.
        \hfill{\small[Thm.~\ref{thm:complex}]}
  \item Solve $\mathcal{L}=0$ \emph{over $\Q$} in the $2nm$ unknowns
        $x_{ij},y_{ij}$; substitute into $Z$ to obtain $\widetilde{Z}$.
  \item \textbf{if} $\widetilde{Z}$ contains no indeterminate
        \textbf{then return} $\widetilde{Z}$.
  \item Expand $\widetilde{Z}A\widetilde{Z}-\widetilde{Z}$, take real and
        imaginary parts, solve the resulting linear system over $\Q$ in the
        surviving unknowns, substitute and \textbf{return}.
\end{enumerate}
\end{algo}

The separation in steps~3, 4 and~6 is not a convenience. The unknowns $x_{ij}$
and $y_{ij}$ have to range over $\Q$, and not over $\Q(i)$, because over $\Q(i)$
the representation $Z=X+iY$ is not unique: the pairs $(X,Y)$ and $(X+iC,\,Y-C)$
give the same $Z$ for every $C$. Besides, the derivation of
\eqref{eq:C1}--\eqref{eq:C4} in Theorem~\ref{thm:complex} uses that the real and
the imaginary part of $AX$ are real. So if we solved over $\Q(i)$ we would
enlarge the solution set in a degenerate direction, and we would lose the
uniqueness that step~6 needs. The $2nm$ real unknowns are the ones that the
splitting produces by itself.

\begin{algo}[H]
\caption{$\textsc{MoorePenroseRat}(A,U)$: rational function entries}
\label{alg:param}
\KwIn{$A\in\Q(u_{1},\dots,u_{\ell})^{m\times n}$ and the parameter list
$U=\{u_{1},\dots,u_{\ell}\}$.}
\KwOut{$\pinv{A}\in\Q(u_{1},\dots,u_{\ell})^{n\times m}$.}
\begin{enumerate}\itemsep2pt
  \item Steps 1--4 of Algorithm~\ref{alg:real}, over
        $\Q(u_{1},\dots,u_{\ell})$.
  \item $\mathcal{M}\leftarrow$ entries of
        $\widetilde{X}A\widetilde{X}-\widetilde{X}$, expanded and normalized;
        \quad $V\leftarrow\operatorname{indets}(\mathcal{M})\setminus U$.
        \hfill{\small[Rem.~\ref{rem:params-are-coeffs}]}
  \item Solve $\mathcal{M}=0$ in $V$, treating $u_{1},\dots,u_{\ell}$ as elements
        of the coefficient field; substitute into $\widetilde{X}$ and
        \textbf{return}.
\end{enumerate}
\end{algo}

\begin{remark}
\label{rem:params-are-coeffs}
We must treat the parameters as elements of the coefficient field, and never as
additional variables of the ideal. Among the variables they change the problem:
the ideal stops being zero-dimensional in the unknowns, and the elimination
degenerates. The difference from Algorithm~\ref{alg:real} is then not the linear
algebra, it is the coefficient arithmetic: Gaussian elimination over
$\Q(u_{1},\dots,u_{\ell})$ works with rational functions whose numerators and
denominators grow and need simplification and normalization. This is why reducing
the number of \emph{variables} from $nm$ to $d$ pays off more in the parametric
setting than in the numeric one, and why the coefficient arithmetic dominates the
cost (Section~\ref{ssec:phases}).
\end{remark}

Examples~\ref{ex:running} and~\ref{ex:runningcont} are a trace of
Algorithm~\ref{alg:param} on a $2\times2$ input over $\R(u)$. Step~1 produces the
line \eqref{eq:runningparam}, where the only surviving indeterminate is $t$, so
$V=\{t\}$. Step~2 expands the outer condition and gives $-t\,Z$, with
$Z=NN^{\mathsf T}$. And step~3 solves $t=0$ over the coefficient field $\R(u)$
and returns $\pinv{A}$.

\begin{proposition}
\label{prop:correctness}
Algorithms~\ref{alg:real}--\ref{alg:param} terminate and return $\pinv{A}$.
\end{proposition}

\begin{proof}
By Theorem~\ref{thm:mpfield} and Corollary~\ref{cor:Ruexists}, $\pinv{A}$ exists
and lies in $K^{n\times m}$ for each of the three input types, so every system
solved below is consistent. Take Algorithm~\ref{alg:real} first. Its step~3
solves a consistent affine-linear system and returns a parametrization of
$A\{1,3,4\}\ni\pinv{A}$. If no indeterminate survives then
$d=0$ and $A\{1,3,4\}=\{\pinv{A}\}$ by Theorem~\ref{prop:dimVL}, so the early
return is correct. Otherwise the expanded entries of
$\widetilde{X}A\widetilde{X}-\widetilde{X}$ are affine-linear and vanish exactly
at $\pinv{A}$ by Theorem~\ref{thm:residuallinear}, so the second system is
consistent with a unique solution.

In Algorithm~\ref{alg:complex} the unknowns range over $\Q$, so a complex
equation vanishes exactly when its real and imaginary parts vanish. Step~3
imposes those parts for \eqref{eq:P1}, and Theorem~\ref{thm:complex} says that
\eqref{eq:C1}--\eqref{eq:C4} are equivalent to \eqref{eq:P3}--\eqref{eq:P4} for
real $X$ and $Y$. So the solution set of the system of step~3 is $A\{1,3,4\}$
written in the real coordinates, and step~6 imposes \eqref{eq:P2} on it as in the
real case. In Algorithm~\ref{alg:param} the only new point is that the parameters
are not among the unknowns that we solve for, so the second system is again a
linear system over the coefficient field $\Q(u_{1},\dots,u_{\ell})$, with the
same unique solution.
\end{proof}

So what we have to count is two linear solves over $K$, of sizes $nm$ and $d$.
Over $\Q$ this is polynomial in the order, and it is dominated by the coefficient
growth of exact Gaussian elimination: a matrix of full rank costs one linear
solve, however large it is, and a matrix of fixed low rank has $d$ growing
quadratically in the order. Over $K=\Q(u_{1},\dots,u_{\ell})$ the count in field
operations is the same, but now a field operation is a gcd and a normalization of
multivariate rational functions whose size grows with the degrees of the entries.
So the cost model that we use is the number of field operations times the cost of
one of them, and it is the second factor that makes the parametric regime
expensive. This is the mechanism of Remark~\ref{rem:params-are-coeffs}.
Section~\ref{sec:degrees} bounds the size of the \emph{answer}, which is a proxy
for the size of the coefficients handled along the way but not a bound on them,
and Section~\ref{sec:cost} measures what it costs.

\section{The cost of the method}
\label{sec:cost}

\subsection{Design}
\label{ssec:design}

All computations used Maple~2025.2 on a 14-core Apple Silicon processor under
macOS, with CPU time measured by \texttt{time()}. Unless we say otherwise, the
matrices of size $m\times(m+1)$ and prescribed generic rank $r$ are $A=FG$, with
$F\in K^{m\times r}$ and $G\in K^{r\times(m+1)}$, integer coefficients sampled
uniformly from $[-10,10]$, and parametric factors of prescribed degree.

The baseline is Maple's \texttt{MatrixInverse(A, method=pseudo)}. On
representative parametric inputs we checked that it satisfies
\eqref{eq:P1}--\eqref{eq:P4} and agrees entry by entry with
Algorithm~\ref{alg:param} and with \eqref{eq:closedentry}, so all the speedups
below compare the same matrix.

The study contains $972$ timed replications: $808$ time the built-in routine
against our method on the same matrix, $70$ pair Algorithm~\ref{alg:param} with
the determinantal formula, $70$ compare three ways of solving the residual system,
and $24$ measure the cost of evaluating the closed form.
The numeric cells use $20$ replications and the parametric ones $10$, except
Tables~\ref{tab:ell} and~\ref{tab:det}, which use five. We time the paired
methods on the same random matrix and with randomized execution order, and we
clear the relevant Maple remember tables before each call. We report the time
limits and the non-completions instead of discarding them, and all the seeds and
the raw timings are in the deposit.

The timings are right-skewed, so the tables report medians and interquartile
ranges. For each paired matrix $A_i$ we use
$s_i=t_{\rm builtin}(A_i)/t_{\rm sym}(A_i)$ and report
$\operatorname{median}(s_i)$, with the wins and the two-sided exact sign test.
Re-timing the same matrices changes the paired speedups by about $10\%$ in median,
and independent samples vary more. One configuration, order $4\times5$ with rank
two, two parameters and factors of degree one, occurs in three of the designs
below on independent samples of ten matrices, with medians $0.105$, $0.228$ and
$0.258$; bootstrapping over $20\,000$ resamples gives overlapping $95$ per cent
intervals, and pooling the thirty matrices gives $0.211$ with interval
$[0.166,0.260]$. So a median of ten carries an interval of roughly a factor of
two, and we only read factors of several as substantive.

\subsection{Numeric and parametric entries}
\label{ssec:results}

With integer entries the built-in routine is uniformly faster. Over the $24$ cells
of the numeric design the paired speedup of the symbolic route goes from $0.50$
down to $0.002$, and the sign test rejects equality at $p<10^{-4}$ in every cell,
so for numeric entries the internal routines are the right tool. Those $480$
replications do verify the theory twice over. The solution space of the first
system had dimension exactly $d=(m-r)(n-r)$ over the coefficient field in the real
regime, as Theorem~\ref{prop:dimVL} requires, and $2d$ real coordinates in the
complex one, which is the count of Remark~\ref{cor:twolinearsolves}. And the
residual system reduced to exactly that many independent linear forms with a
single solution, the signature of Corollary~\ref{rem:zerodim}.

With parametric entries the picture reverses. The parametric cells reproduce and
extend the earlier experiments with the same generator: rank-two matrices of
order $m\times(m+1)$ for $m=4,\dots,10$, with entries rational in three parameters
and ten replications per order. We keep it
unchanged, so that the new orders stay comparable with the earlier ones, but its
structure is not imposed by design and is worth describing in full. Write $\delta$
for the entry degree. Row one is filled with $n$ random polynomials of degree at
most $\delta$, and we draw two more, $f_{1}$ and $f_{2}$, of the same degree,
replaced by $1$ if they come out zero. Then, \emph{for the first $m$ columns
only},
\[
  A_{2j}=\frac{A_{1j}}{f_{1}},
  \qquad
  A_{3j}=\frac{A_{1j}}{f_{1}}\Bigl(1+\frac{1}{f_{2}}\Bigr),
  \qquad 1\le j\le m,
\]
with rows $4,\dots,m$ left at zero. Rows two and three are then proportional, and
the last $n-m$ columns of those rows stay zero while row one does not, so the rank
is exactly two: what produces it is that the column loop stops at $m$ and not at
$n$. The family is therefore degenerate for a reason that has nothing to do with
the mathematics of Sections~\ref{sec:theory} and~\ref{sec:degrees}. It is a
legitimate family of rank-two parametric matrices all the same, and we checked on
the generated matrices that one row is identically zero, that the rank is two, and
that the residual system has total degree one throughout.

We did not assume how much the construction matters, we measured it. At order
$4\times5$, rank two, two parameters and ten matrices per cell, the generic
generator $A=FG$ gives median paired speedups of $0.258$ at factor degree one,
$0.237$ at factor degree five and $0.267$ with rational entries, while the
structured family gives $0.476$ and $2.680$ at the same two degrees. At factor
degree five it is faster on eight matrices of ten, against none of ten for the
generic generator, and Fisher's exact test gives $p=0.0007$. Neither degeneracy
nor high degree is enough on its own, and their combination is, because degeneracy
holds the monomials per entry to six where the generic construction reaches
forty-two, and the size of the coefficients is the binding constraint of
Section~\ref{ssec:phases}. The ten speedups in that cell are trimodal, so the
direction of the finding is secure and its magnitude is not.

\begin{table}[htbp]
\centering
\caption{Structured parametric family: rank-two matrices with entries rational in
three parameters, ten matrices per order. Median CPU time in seconds, median
paired speedup of the symbolic route against the built-in routine with its
interquartile range, wins, sign test and $d$. Conventions as in
Section~\ref{ssec:design}; on the two implementations behind this table and
Table~\ref{tab:phases}, see Remark~\ref{rem:harnesses}.}
\label{tab:param}
\small
\begin{tabular}{lrrrcrlr}
\toprule
Order & Built-in & Symbolic & Speedup & IQR & Wins & $p$ & $d$ \\
\midrule
$4\times5$   &  63.1 &  18.7 & 3.45 & $[2.70,\,4.28]$ & 10/10 & 0.002 &  6 \\
$5\times6$   &  81.7 &  33.7 & 2.55 & $[2.28,\,3.02]$ & 10/10 & 0.002 & 12 \\
$6\times7$   & 114.9 &  45.4 & 2.43 & $[2.12,\,2.72]$ & 10/10 & 0.002 & 20 \\
$7\times8$   & 131.2 &  60.7 & 1.96 & $[1.83,\,2.45]$ & 10/10 & 0.002 & 30 \\
$8\times9$   & 134.5 &  78.6 & 1.80 & $[1.54,\,2.06]$ & 10/10 & 0.002 & 42 \\
$9\times10$  & 155.3 & 100.7 & 1.49 & $[1.35,\,1.73]$ & 10/10 & 0.002 & 56 \\
$10\times11$ & 181.3 & 196.9 & \textbf{0.89} & $[0.87,\,0.96]$ & 0/10 & 0.002 & 72 \\
\midrule
pooled & --- & --- & 1.93 & $[1.48,\,2.51]$ & 60/70 & $8\times10^{-10}$ & --- \\
\bottomrule
\end{tabular}
\end{table}

\begin{figure}[htbp]
\centering
\includegraphics[width=0.72\linewidth]{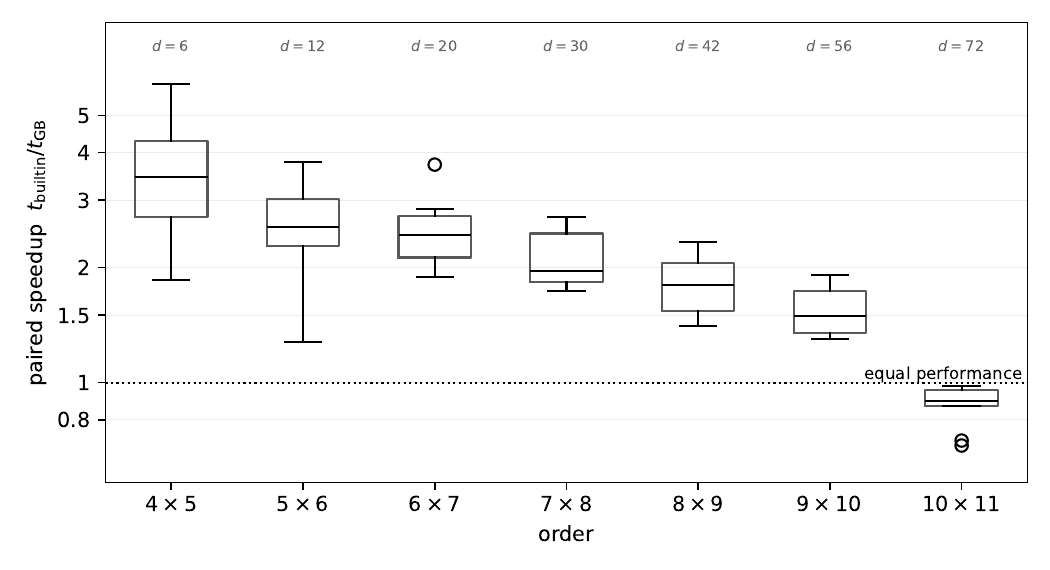}
\caption{Distribution of the per-matrix paired speedup on the structured parametric
family, ten matrices per order, logarithmic scale. The dotted line is equal
performance.}
\label{fig:cruce}
\end{figure}

Every replication passed the verification of the four Penrose equations, and none
of them exceeded the time limit. The effect is not constant. The advantage falls
with the order and it crosses equal performance between the last two cells, and
Figure~\ref{fig:cruce} shows the per-matrix spread behind each cell. The sequence
decreases at every step, but two of the six steps are of five and ten per cent, so
they lie inside the measurement error of Section~\ref{ssec:design}. What the data
support is the trend and the crossing, and not monotonicity cell by cell. At
$10\times11$ the interquartile range lies entirely below one and the largest
speedup that we observed in that cell is $0.98$, so we have measured the reversal
and we are not extrapolating it.

This experiment alone cannot say whether the decline comes from the order or from
the residual dimension $d=(m-2)(m-1)$, because inside this family the two grow
together. The pairs $(nm,d)$ that the seven cells realize are
$(20,6),(30,12),(42,20),(56,30),(72,42),(90,56),(110,72)$, and the two candidate
abscissae are collinear. To separate them we have to vary $d$ at a fixed order,
and that is the design of the rank sweep below.

\subsection{What drives the cost}
\label{ssec:phases}

Algorithm~\ref{alg:param} has three phases: the first linear solve, the formation
and expansion of $\widetilde{X}A\widetilde{X}-\widetilde{X}$, and the elimination
that finishes it. Table~\ref{tab:phases} times them separately, and this tells us
what would have to change for the method to become faster.

\begin{table}[htbp]
\centering
\caption{Phase decomposition of Algorithm~\ref{alg:param} on the structured
family, medians over ten replications, in seconds. $T_{\mathrm{lin1}}$ is the
first linear solve, $T_{\mathrm{exp}}$ the formation and expansion of
$\widetilde{X}A\widetilde{X}-\widetilde{X}$, $T_{\mathrm{elim}}$ the elimination,
and the built-in column is for scale.}
\label{tab:phases}
\small
\begin{tabular}{lrrrrrr}
\toprule
Order & $d$ & built-in & $T_{\mathrm{lin1}}$ & $T_{\mathrm{exp}}$
      & $T_{\mathrm{elim}}$ & total \\
\midrule
$4\times5$ &  6 &  52.7 &  2.02 &  4.54 & 1.17 &  8.02 \\
$5\times6$ & 12 &  83.8 &  4.93 & 11.81 & 2.33 & 19.38 \\
$6\times7$ & 20 & 110.8 &  7.44 & 17.77 & 3.73 & 27.48 \\
$7\times8$ & 30 & 128.1 & 11.80 & 25.30 & 5.20 & 40.15 \\
$8\times9$ & 42 & 143.2 & 13.60 & 35.50 & 5.94 & 56.22 \\
\bottomrule
\end{tabular}
\end{table}

\begin{remark}[implementation effect]
\label{rem:harnesses}
Tables~\ref{tab:param} and~\ref{tab:phases} use the same algorithm, but they
normalize the residual in different ways: \texttt{simplify} in the first one, and
\texttt{expand} followed by \texttt{normal} in the second one. On fifteen paired
matrices the second one was faster by a median factor $1.54$, and it returned the
same $\pinv{A}$.
\end{remark}

Across the five cells the residual expansion takes a median $55$--$68\%$ of the
total symbolic time, the first solve $25$--$31\%$, and the final elimination only
$8$--$15\%$. We also compared three ways of solving the residual system, a
Gr\"obner basis, a general-purpose solver and coefficient-matrix extraction: no
strategy was consistently the fastest, and in the polynomial cells the final
elimination stayed below $14\%$ of the symbolic total with all three of them. So
the bottleneck is the manipulation of the coefficients, in the first solve and in
the construction of the residual, and not the final elimination.

We turn now to the rank. The experiments so far vary the order and keep the rank
at two, so they cannot separate the two effects. Here we fix the order, the
number of parameters and the entry degree, and vary only the rank: matrices of order $4\times5$ as $A=FG$, with
every entry of $F$ and of $G$ an independent random polynomial of degree one in
two parameters, so that $A$ has generic rank $r$ and entries of degree two, and
$r=4,3,2$, which gives $d=0,2,6$, with ten matrices per cell. The construction
invites a confound, because as $r$ grows each entry becomes a sum of $r$ products.
But the median number of monomials per entry is $6$ in the three cells, the
dimension of the space of bivariate polynomials of degree two, so the support
saturates already at $r=2$. The rank is then the designed factor and the monomial
support stays approximately constant, although the coefficient distributions are
not identical across cells and we do not claim that the rank is literally the only
difference.

\begin{table}[htbp]
\centering
\caption{Order $4\times5$, two parameters, entry degree two; only the rank
varies. Median CPU times in seconds over ten matrices per cell.
Conventions as in Section~\ref{ssec:design}.}
\label{tab:rank}
\small
\begin{tabular}{rrrrrcrl}
\toprule
$r$ & $d$ & built-in & symbolic & speedup & IQR & Wins & $p$ \\
\midrule
4 &  0 &  24.99 & 12.27 & $1.81$ & $[1.57,\,2.68]$ & 10/10 & 0.002 \\
3 &  2 &  13.86 & 19.58 & $0.70$ & $[0.47,\,0.75]$ &  0/10 & 0.002 \\
2 &  6 &   0.64 &  6.42 & $0.10$ & $[0.10,\,0.16]$ &  0/10 & 0.002 \\
\bottomrule
\end{tabular}
\end{table}

Table~\ref{tab:rank} shows a speedup that falls by a factor of seventeen on the
unrounded medians and crosses equal performance between full rank and rank three,
with no exception in the thirty matrices. What drives the change is the baseline:
from $r=4$ to $r=2$ the median built-in time falls from $24.99$ to $0.64$
seconds, while the symbolic times stay in a much narrower range and are not even
monotone in the rank. So rank deficiency benefits the built-in routine much more than it benefits ours. Between
these cells $nm$ is constant while $d$ varies, and the symbolic time varies with
it, so $d$ does carry information that $nm$ does not, although only by a factor of
$1.6$.

Is it the rank or is it the residual dimension? Neither that sweep nor the
parametric series of Section~\ref{ssec:results} separates the rank from $d$,
because at a fixed rank the two move together. We can separate them, because
$d=(n-r)(m-r)$ depends on the order and on the rank only through $m-r$ and $n-r$:
in shape $m\times(m+1)$, if we take $r=m-2$ then $d=6$ for \emph{every} $m$, so we
can raise the order and keep the residual system at six unknowns. At $m=4$ this
prescription returns $r=2$, which is the rank-two cell itself, so the two series
of Table~\ref{tab:dfijo} begin on the same matrices, and that shared cell is a
control on the design and not a separate measurement.

\begin{table}[htbp]
\centering
\caption{Two series over the same orders, $A=FG$ with factors of degree one in
two parameters, so $\delta=2$. In the first $d=6$ is fixed and the rank rises with
the order; in the second the rank is two and $d$ rises. Median CPU times in
seconds over ten matrices per cell, conventions as in
Section~\ref{ssec:design}.}
\label{tab:dfijo}
\small
\begin{tabular}{lrrrrrl}
\toprule
& Order & $r$ & $d$ & built-in & symbolic & speedup \\
\midrule
\multirow{3}{*}{$d=6$ fixed}
 & $4\times5$ &  2 &  6 &   0.9 &    3.7 & $0.228$ \\
 & $5\times6$ &  3 &  6 &  17.2 &  103.8 & $0.170$ \\
 & $6\times7$ &  4 &  6 &  73.9 & 1946.8 & $0.036$ \\
\midrule
\multirow{5}{*}{rank $2$ fixed}
 & $4\times5$ &  2 &  6 &   0.9 &    3.7 & $0.228$ \\
 & $5\times6$ &  2 & 12 &   4.1 &   20.2 & $0.216$ \\
 & $6\times7$ &  2 & 20 &   1.5 &   47.4 & $0.032$ \\
 & $7\times8$ &  2 & 30 &   3.2 &  113.6 & $0.026$ \\
 & $8\times9$ &  2 & 42 &   6.0 &  205.1 & $0.020$ \\
\bottomrule
\end{tabular}
\end{table}

The two series answer the question. Multiplying the residual dimension by seven
multiplies the symbolic time by $55$; raising the rank from two to four with the
residual dimension \emph{fixed at six} multiplies it by $524$. The comparison is
sharpest inside a single order, and we take it at $6\times7$, where both cells
carry ten replications: there the rank-four matrix with $d=6$ takes $1946.8$
seconds, against $47.4$ for the rank-two matrix with $d=20$, a factor of $41$ in
favour of the larger residual system. Neither series
varies one factor at a time, since in the first one the rank rises with the order
and in the second one $d$ does, so this is not a causal separation of order, rank
and $d$. What it shows is that at a fixed order the rank dominates the residual
dimension by a wide margin.

This agrees with Theorem~\ref{thm:bound}, which bounds degrees and not running
times. Its bounds involve the rank and the entry degree, and neither $m$ nor $n$
nor $d$, so along the rank-two series the degrees stay at $3\delta$ and $4\delta$
however large the matrix is, while along the $d=6$ series they grow to $7\delta$
and $8\delta$ at $r=4$. The residual dimension counts the free variables of the
second system, the rank governs the degrees of the rational functions that we
manipulate, and the measurement says that the second effect dominates. We control
the size of the entries here too: the median number of monomials per entry of $A$
is $6$ in the seventy-four replications.

\begin{remark}[the largest cell of the $d=6$ series]
\label{rem:dfijo78}
At $7\times8$ with $r=5$ we measured the $d=6$ series on four matrices and not on
ten, and we stopped the run because the next cell would have needed a rank-six
matrix. The median symbolic time, $9035$ seconds, continues the trend of the
previous cells. But with four replications and a fivefold spread it is too
variable for a quantitative comparison, so we report it here and not in
Table~\ref{tab:dfijo}.
\end{remark}

The last factor is the number of parameters. Every table so far fixes $\ell$ at
two or three. The cost
model of Section~\ref{sec:algorithms} attributes the expense to the arithmetic
of the coefficients in $\Q(u_{1},\dots,u_{\ell})$, so $\ell$ should matter, and
the degree bound of Theorem~\ref{thm:bound} says in what way. We now vary it
alone. We take order $4\times5$, rank two, and $A=FG$ with the entries of $F$ and
$G$ of degree one, so that the entries of $A$ have degree at most $\delta=2$. We
take $\ell=1,2,3$ and five matrices per cell, with $d=6$ in the fifteen
replications, so the system that the second solve handles has the same size
throughout.

\begin{table}[htbp]
\centering
\caption{Order $4\times5$, rank two, factors of degree one and hence $\delta=2$;
only $\ell$ varies, with $d=6$ throughout. Median CPU times in seconds over five
matrices per cell. The monomial columns are median counts per entry, for $A$ and
for the numerator of $\pinv{A}$.}
\label{tab:ell}
\small
\begin{tabular}{rrrrrrrl}
\toprule
$\ell$ & mon $A$ & mon $\pinv{A}$ & built-in & $T_{\mathrm{coef}}$
& $T_{\mathrm{elim}}$ & symbolic & speedup \\
\midrule
1 &  3 &  7 &  0.15 &  0.42 & 0.06 &  0.48 & $0.315$ \\
2 &  6 & 28 &  0.71 &  2.42 & 0.46 &  2.88 & $0.271$ \\
3 & 10 & 84 & 15.22 & 34.78 & 3.88 & 38.85 & $0.392$ \\
\bottomrule
\end{tabular}
\end{table}

The monomial counts are the point of the table. On this family they are not
approximately what Theorem~\ref{thm:bound} allows, they are exactly it. Here $r=2$
and $\delta=2$, so the theorem bounds the numerator of $\pinv{A}$ by total degree
$(2r-1)\delta=6$, and the number of monomials of total degree at most $D$ in
$\ell$ variables is $\binom{D+\ell}{\ell}$, so a dense numerator attaining the
bound would have $7$, $28$ and $84$ terms. Those are the measured counts, with no
deviation in the fifteen replications, and the entries of $A$ realize
$\binom{\delta+\ell}{\ell}=3,6,10$ in the same way. So here the bound is attained,
the numerator is dense, and the closed form has
$\binom{(2r-1)\delta+\ell}{\ell}$ terms per entry. The bound does not involve
$\ell$ at all, and that is precisely the point: the rank fixes the degree, and the
number of parameters fixes how many coefficients a polynomial of that degree can
carry.

With $d$ fixed at six the median symbolic time grows from $0.48$ to $38.85$
seconds between one and three parameters, and the built-in routine slows down by a
comparable factor. The dimension of the second system does not change; what grows
is the size of the coefficients that it manipulates, and the coefficient phase
takes then between $84$ and $90$ per cent of the symbolic total, which reproduces
the decomposition of Section~\ref{ssec:phases} on a different family. Across the
three cells the paired speedup shows no systematic trend, although with five
replications per cell we would not detect a small one. So the number of parameters
makes the problem harder for both methods rather than changing the comparison
between them.

\subsection{The determinantal formula as an algorithm}
\label{ssec:det}

We introduced the determinantal representation of
Proposition~\ref{prop:closedform} in order to obtain the degree bound. But it is
also a way of computing $\pinv{A}$, and in this subsection we time it against the
algorithm of Section~\ref{sec:algorithms} on the same matrices.

A direct evaluation needs $\binom{m}{r}\binom{n}{r}$ minors of order $r$ and, for
each one of them, $r^{2}$ cofactors of order $r-1$. A simple measure of that
combinatorial workload is
\begin{equation}
  \label{eq:blocks}
  \mathcal{B}(m,n,r)\;:=\;\binom{m}{r}\binom{n}{r}\,r^{2},
\end{equation}
and it is not monotone in $r$, because the binomial factors collapse at full rank.
So $\mathcal{B}$ rises, it peaks in the middle of the rank range, and it falls
again. This shape is different from anything that we have measured so far, since
by the rank sweep of Section~\ref{ssec:phases} the cost of the linear route rises
with the rank all the way.

The design sweeps the whole rank range at two orders, $6\times7$ and $8\times9$,
with five matrices per rank, one parameter, and $A=FG$ with the entries of $F$ and
$G$ random of degree one, so that $\delta=2$. The two arms are the elimination of
Algorithm~\ref{alg:param}, with the \texttt{expand}$\,\circ\,$\texttt{normal}
expansion of Remark~\ref{rem:harnesses}, and \eqref{eq:closedentry} evaluated
directly, drawn in random order in every replication. In every replication the two
arms returned the same matrix, checked entry by entry.

\begin{table}[htbp]
\centering
\caption{One parameter, $\delta=2$, generic $A=FG$; medians over five matrices per
cell, in seconds. $\mathcal{B}$ is \eqref{eq:blocks}, in bold at the rank where it
peaks in each block. ``Ratio'' is the median of the per-matrix ratios and
``mon.'' the number of monomials in the numerator of an entry of $\pinv{A}$,
constant within every cell.}
\label{tab:det}
\small
\begin{tabular}{rrrrrrrr}
\toprule
Order & $r$ & $\mathcal{B}$ & $d$ & Algorithm~\ref{alg:param} & \eqref{eq:closedentry}
& Ratio & mon. \\
\midrule
$6\times7$ & 1 &   42 & 30 &  0.38 & 0.001 & 352.0 &  3 \\
           & 2 & 1\,260 & 20 &  3.99 & 0.104 &  37.9 &  7 \\
           & 3 & 6\,300 & 12 & 13.15 & 0.823 &  16.1 & 11 \\
           & 4 & \textbf{8\,400} &  6 & 18.09 & 2.007 & \textbf{7.7} & 15 \\
           & 5 & 3\,150 &  2 & 31.19 & 1.154 &  21.8 & 19 \\
           & 6 &    252 &  0 & 31.24 & 0.089 & 245.2 & 17 \\
\midrule
$8\times9$ & 1 &      72 & 56 &   0.78 &  0.002 & 391.0 &  3 \\
           & 2 &   4\,032 & 42 &  12.35 &  0.389 &  31.2 &  7 \\
           & 3 &  42\,336 & 30 &  79.16 &  8.890 &   8.8 & 11 \\
           & 4 & 141\,120 & 20 & 240.57 & 45.121 &   5.3 & 15 \\
           & 5 & \textbf{176\,400} & 12 & 308.89 & 73.589 & \textbf{4.4} & 19 \\
           & 6 &  84\,672 &  6 & 432.81 & 79.453 &   5.6 & 23 \\
           & 7 &  14\,112 &  2 & 546.23 & 78.776 &   7.3 & 27 \\
           & 8 &     576 &  0 & 465.99 &  4.830 &  74.5 & 23 \\
\bottomrule
\end{tabular}
\end{table}

Three things are visible in Table~\ref{tab:det}. First, the formula is faster at
every rank of both orders. The smallest factor that we observed in the seventy
replications is $2.84$ and the largest one is $1301$. Second, the two arms are different in kind, and not only in size. The time of
Algorithm~\ref{alg:param} increases with the rank over almost the whole range,
from $0.38$ to $31.2$ seconds at $6\times7$ and from $0.78$ to $546$ at
$8\times9$ up to $r=7$, with one single exception at $r=8$, where it falls back to
$466$. The time of the formula, instead, is unimodal, and it peaks at $r=4$ at
$6\times7$ and at $r=6$ at $8\times9$, one rank beyond the peak of
$\mathcal{B}$ at the larger order. So what the cost model locates exactly at both
orders is not the peak of the time, it is the \emph{minimum of the ratio}: the
last column of Table~\ref{tab:det} falls and then rises, with its minimum at
$r=4$ at $6\times7$ and at $r=5$ at $8\times9$, in both cases exactly where
$\mathcal{B}$ is largest. Third, full rank is the case where Algorithm~\ref{alg:param} has the least to do,
since there $d=0$ and it returns at its first solve. And the formula still wins
there by $245$ at $6\times7$ and by $74.5$ at $8\times9$, because $\mathcal{B}$
collapses to $252$ and $576$. So full rank is favourable to both arms at the same
time, and it is the one cell of the table that should not be read as a comparison
under stress.

The last column is an independent check on the parameter sweep of
Section~\ref{ssec:phases}. With $\delta=2$ and $\ell=1$ the count $\binom{(2r-1)\delta+\ell}{\ell}$ that we
established there equals $4r-1$, and that is the measured value at every rank
below the full one, at both orders and without exception, on a family and a
parameter count that that section did not use. So the count depends on the rank
and on the entry degree, and not on the order, which is what
Theorem~\ref{thm:bound} says, and the bound is attained.

At full rank it is not attained, and the shortfall is a definite amount: the bound
gives $4m-1$, that is, $23$ at $6\times7$ and $31$ at $8\times9$, while the
measured values are $17$ and $23$, in both cases $3m-1$. This is the criterion of
Remark~\ref{rem:tight} acting through the generator. At $r=m$ the factor $F$ of
$A=FG$ is square, so $\det F$ divides every maximal minor and hence $c(A)$, and it
has degree $m$ when the factors have degree one, which removes exactly $m$ from
the bound.

So the advantage narrows with the order where $\mathcal{B}$ is large, and
$\mathcal{B}$ grows fast: its maximum over $r$ is $8\,400$ at $6\times7$,
$176\,400$ at $8\times9$ and $3.49\times10^{6}$ at $10\times11$. A last difference
is not about speed at all. The formula needs $r$, while
Algorithm~\ref{alg:param} recovers $d=(m-r)(n-r)$ from its first solve, and in the
parametric setting the rank is a function of the parameters that drops exactly on
the exceptional set. So a caller who has not established the rank cannot use
\eqref{eq:closedentry} directly, and Corollary~\ref{thm:excep} says which
specializations to exclude. In that respect the two routes are complementary.

Two limits remain, and we state them. These are two orders, one parameter, one
entry degree, one generator and five replications per cell, so the design
establishes that the formula is competitive at these orders and not how
competitive it is in general. And it does not establish that $\mathcal{B}$
\emph{governs} the cost, because an implementation can reuse minors across pairs
$(I,J)$ and the cost of each determinant depends on the coefficients too. What the
measurements support is that $\mathcal{B}$ captures the combinatorial shape of the
cost and locates the minimum of the ratio exactly.

Three regimes come out of all this. With numeric entries the built-in routine is
the right tool. With parametric entries both routes give a closed form, and which
one is faster depends on the size of the entries and on the family, and not on the
rank alone. And when the rank is known, \eqref{eq:closedentry} is worth trying,
computing $\mathcal{B}(m,n,r)$ first as an indicator of the workload.

\begin{remark}[obtaining the closed form and evaluating it are different costs]
\label{rem:evaluar}
Nothing above says that the closed form is cheap to \emph{use}. Specializing it at
a parameter value, instead of solving again numerically there, is cheaper by
factors of $3.80$, $3.04$, $1.73$ and $1.23$ at orders $4\times5$ to $7\times8$
with one parameter, so the saving is gone by the last of them, and with two
parameters at $4\times5$ it reverses and specializing costs about ten times more.
The mechanism is Theorem~\ref{thm:bound}: the closed form carries
$\binom{(2r-1)\delta+\ell}{\ell}$ terms per entry while $A$ stays small. So the
case for computing $\pinv{A}$ symbolically does not rest on repeated evaluation.
\end{remark}

\section{Application to a parametric Leontief model}
\label{sec:application}

In the Leontief input--output model \cite{Leontief1951} an economy of $n$ sectors is
described by
\begin{equation}
  \label{eq:leontief}
  x = Ax + f, \qquad\text{equivalently}\qquad (I-A)x = f,
\end{equation}
where $x\in\R^{n}$ is the vector of total sectoral output, $f\in\R^{n}$ is final
demand, and $A=(a_{ij})$ is the matrix of technical coefficients, so that $a_{ij}$
is the amount of the output of sector $i$ that we need in order to produce one
unit of the output of sector $j$. The model is \emph{productive} when
$(I-A)^{-1}$ exists and is non-negative, and by the Hawkins--Simon condition
\cite{HawkinsSimon1949} this happens in particular when every column sum of $A$ is
smaller than one.

Technical coefficients are estimated quantities. If we treat one of them as a
parameter, \eqref{eq:leontief} becomes a system over $\Rat$, and then
Theorem~\ref{thm:specialization} says that we can obtain every nonexceptional
specialization from one single symbolic computation. We take a three-sector
economy with
\begin{equation}
  \label{eq:Aexample}
  A(u)=
  \begin{pmatrix}
    \tfrac{1}{5}  & \tfrac{3}{20}          & \tfrac{1}{20}\\[2pt]
    \tfrac{3}{10} & \tfrac{7}{20}(1+u)     & \tfrac{1}{5}\\[2pt]
    \tfrac{1}{10} & \tfrac{1}{5}           & \tfrac{1}{4}
  \end{pmatrix},
\end{equation}
where the own-input coefficient of the second sector carries a proportional shock
$u$, so that $u=0$ gives the unperturbed economy. At $u=0$ the column sums are
$3/5$, $7/10$ and $1/2$, so the Hawkins--Simon condition holds and the economy is
productive. We fix final demand at $f=(100,150,120)^{\mathsf T}$.

We apply Algorithm~\ref{alg:param} to $I-A(u)$ and we obtain
$\bigl(I-A(u)\bigr)^{\dagger}$ over $\Q(u)$ in one single computation, and with
it $x(u)$ in closed form:
\begin{equation}
  \label{eq:Xclosed}
  x(u)=\frac{1}{7\,(119u-180)}
  \begin{pmatrix}
    100\,(1134u-2825)\\[2pt]
    -539000\\[2pt]
    200\,(742u-1915)
  \end{pmatrix},
\end{equation}
a vector of rational functions of $u$, as Theorem~\ref{thm:mpfield} predicts.
Specializing at $u=0$ gives
\[
  x(0)=\Bigl(\tfrac{14125}{63},\ \tfrac{3850}{9},\
  \tfrac{19150}{63}\Bigr)^{\mathsf T}
  \approx(224.206,\ 427.778,\ 303.968)^{\mathsf T}.
\]
And since we have $x(u)$ in closed form, we obtain the sensitivity of the output
to the coefficient by differentiating, and not by solving the system again on a
grid:
\[
  \frac{\partial x}{\partial u}\Big|_{u=0}
  =\Bigl(\tfrac{18865}{324},\ \tfrac{45815}{162},\
  \tfrac{13475}{162}\Bigr)^{\mathsf T}
  \approx(58.225,\ 282.809,\ 83.179)^{\mathsf T},
\]
\[
  \frac{\partial \log x_{i}}{\partial u}\Big|_{u=0}
  =\Bigl(\tfrac{26411}{101700},\ \tfrac{119}{180},\
  \tfrac{3773}{13788}\Bigr)
  \approx(0.2597,\ 0.6611,\ 0.2736).
\]
The second sector is the most responsive to a shock in its own input requirement,
as we expected. These are exact rational numbers and not finite-difference
approximations, so there is no step size to choose and no truncation error.

We compute the determinant symbolically,
\begin{equation}
  \label{eq:Delta}
  \Delta(u)=\det\bigl(I-A(u)\bigr)
  =\frac{63}{200}-\frac{833}{4000}\,u
  =-\frac{7\,(119u-180)}{4000},
\end{equation}
and its only root is
\[
  u^{*}=\frac{180}{119}\approx 1.5126,
  \qquad\text{at which}\qquad
  a_{22}(u^{*})=\frac{299}{340}\approx 0.8794 .
\]
At $u=u^{*}$ the rank of $I-A(u)$ drops to $2$. The first two leading principal
minors are still positive there, $4/5$ and $7/136$, and only the determinant
vanishes. So $u^{*}$ is exactly the point where the Hawkins--Simon condition fails
for the first time, that is, the viability frontier of this parametrized economy.
A vanishing determinant does not leave \eqref{eq:leontief} unsolvable by itself,
only without a unique solution. Here the system is in fact inconsistent, because
\[
  \rank\bigl(I-A(u^{*})\bigr)=2
  \qquad\text{while}\qquad
  \rank\bigl(\left[\,I-A(u^{*})\mid f\,\right]\bigr)=3 ,
\]
so no output vector meets that demand exactly, and
$\bigl(I-A(u^{*})\bigr)^{\dagger}f$ is the minimum-norm least-squares solution.
Among the admissible parameter values, $u^{*}$ is exactly the exceptional set
$\mathcal{E}(I-A)$ of Definition~\ref{def:exceptional}, and we can see it in
\eqref{eq:Xclosed} as a pole of the three components: the closed form carries its
own degeneracy locus with it.

We do not need to find it that way, and this is where Corollary~\ref{thm:excep}
pays. The matrix $I-A(u)$ has rank three, so its only maximal minor is $\Delta(u)$
and $S(I-A)=\Delta(u)^{2}$, and the corollary gives
$\mathcal{E}(I-A)\cap\R=\{\Delta=0\}=\{u^{*}\}$. So we read the viability
frontier from one determinant of $A$, before forming any pseudoinverse, and in an
economy of $n$ sectors with several parametrized coefficients that is the
difference between a determinant and a symbolic inversion. A finite grid of
scenarios need not contain $u^{*}$, and unless we choose it so, it shows the
solver becoming ill-conditioned nearby without returning the value itself, and
without a criterion to separate genuine degeneracy from floating-point error.

This example does not exercise Theorem~\ref{thm:residuallinear}, though. Over
$\R(u)$ the matrix $I-A(u)$ has generic rank three, so $d=0$, the pseudoinverse is
the inverse, and Algorithm~\ref{alg:param} returns at its early exit without
forming any residual system. The residual only becomes relevant if we specialize
at $u=u^{*}$ first and compute the pseudoinverse of the resulting rank-two matrix,
where $d=1$ and the residual system is one single affine-linear equation, as the
theorem says in advance.

\begin{remark}
\label{ssec:minnorm}
The minimum-norm selection is a mathematical criterion and not an economic
identification argument. It singles out one element of the solution set of the
degenerate system, and it does not assert that an economy would realize that
element. We offer the example as an illustration of the algebra and not as an
empirical claim.
\end{remark}

\section{Conclusions and future work}
\label{sec:conclusions}

We have shown that, over a Moore--Penrose field, the Penrose system reduces to two
affine-linear solves: on $A\{1,3,4\}$ the outer residual is
$XAX-X=\pinv{A}-X$. This gives a unique linear residual system and a maximal residual
ideal, while the complex case is handled by an equivalent real splitting.

For polynomial matrices the determinantal representation yields
$\deg\num(\pinv{A})\le(2r-1)\delta$ and
$\deg\den(\pinv{A})\le2r\delta$. More importantly, its intrinsic common denominator
$S(A)$ determines the exceptional set directly from $A$. That the zero set of $S(A)$
is the rank-drop locus is elementary; what Corollary~\ref{thm:excep} adds is that
specialization fails at exactly those points, so the locus where the parametric
closed form ceases to represent the specialized pseudoinverse is known before that
pseudoinverse is computed.

The experiments show that parametric cost is dominated by coefficient manipulation,
not by the final residual solve, and that performance depends strongly on rank and on
the matrix family. The determinantal formula outperformed Algorithm~\ref{alg:param}
at every rank of the two tested orders, although its broader regime remains open.
The Leontief example illustrates the main benefit of symbolic computation: exact
sensitivities and an exact viability frontier rather than uniform speed.

Natural next steps are to construct $A\{1,3,4\}$ directly from kernel bases, combine
the degree bounds with interpolation at larger orders, compare the determinantal and
linear routes systematically, and extend the application to rectangular supply--use
systems.


\section*{CRediT authorship contribution statement}

\textbf{J. Sendra:} Conceptualization, Methodology, Formal analysis,
Investigation, Data curation, Software, Validation, Visualization,
Writing --- original draft, Writing --- review and editing.

\section*{Acknowledgements}

Part of this work builds on the final-year project of M.\ Coronado at the Escuela
Polit\'ecnica Superior of CUNEF Universidad, supervised by J.~R.\ Sendra.
Theorem~\ref{thm:complex}, the real splitting of the hermitian conditions, is hers,
as are the first Maple implementations of
Algorithms~\ref{alg:real}--\ref{alg:param} and the measurements that
Section~\ref{ssec:results} reproduces and Section~\ref{ssec:phases} reanalyses. I am
grateful to her for that work and for agreeing that it be developed and published
here.

The theory of Moore--Penrose fields on which this paper rests, and the formulation of
the Penrose system it takes up, originate in earlier joint work of J.~R.\ Sendra and
the present author \cite{SendraSendra2017}, continuing the line of research begun in
\cite{SendraSendra2015,SendraSendra2016}. J.~R.\ Sendra also supervised the work in
which the Maple implementations used here were first developed, and I thank him for
that guidance and for valuable discussions on this material.

\section*{Declaration of competing interest}

The author declares that she has no known competing financial interests or personal
relationships that could have appeared to influence the work reported in this
paper.

\section*{Data availability}

All data are generated by the code accompanying this paper; none is third-party.
The Maple harnesses, the raw per-replication timings, the scripts that produce
every table and figure, and the scripts that verify the examples of
Section~\ref{sec:theory}, the closed form of
Proposition~\ref{prop:closedform} over $\Q(i)$ and every number of
Table~\ref{tab:det} are deposited at
\url{https://doi.org/10.5281/zenodo.23164326}, together with a note recording which
harness and which data file produce each table. Replications that exceeded the time limit are recorded rather than discarded,
including the four of Section~\ref{ssec:design}. The parametric cells of
Section~\ref{ssec:results} reproduce, and the phase decomposition of
Section~\ref{ssec:phases} reanalyses, measurements first obtained by M.\ Coronado in
her final-year project, used here with her permission; those raw timings are included
in the deposit.

\section*{Declaration of generative AI and AI-assisted technologies in the
writing process}

During the preparation of this work the author used Claude (Anthropic) to assist with
drafting and language editing of the text, with \LaTeX\ formatting, and as a checking
aid while revising the manuscript and verifying the bibliographic details of the cited
sources against the original publications. The mathematical content is the author's own: the
statements, proofs, algorithms and bounds, the design of the experiments and the
choice of which further experiments to run, the interpretation of the measurements and
the conclusions drawn from them. Every computation reported here was performed and
checked by the author. After using this tool the author reviewed and edited the content as
needed and takes full responsibility for the content of the publication.

\ifelsevier
  \bibliographystyle{elsarticle-num}
\else
  \bibliographystyle{plain}
\fi
\ifelsevier\else\small\setlength{\itemsep}{0pt}\fi

\end{document}